\documentclass{birkjour}

\usepackage{amsmath, amssymb}
\usepackage{amsfonts}
\usepackage{graphicx}
\usepackage{mathrsfs}
\usepackage[arrow,matrix,curve,cmtip,ps]{xy}
\usepackage{amsthm}
\usepackage{enumerate}
\usepackage{color}
\usepackage[colorlinks]{hyperref}
\usepackage{tikz}

 \newtheorem{thm}{Theorem}[section]
 \newtheorem{cor}[thm]{Corollary}
 \newtheorem{lem}[thm]{Lemma}
 \newtheorem{prop}[thm]{Proposition}
 \theoremstyle{definition}
 \newtheorem{defn}[thm]{Definition}
 \theoremstyle{remark}
 \newtheorem{rem}[thm]{Remark}
 \newtheorem{example}[thm]{Example}
 \numberwithin{equation}{section}

\newcommand{\g}{\mathcal{G}_{G,\Lambda}}

\begin{document}

%
%
%
%
%
%
%
%
%

\title[Minimality and effectiveness]{Minimality and effectiveness of the groupoid associated to a self-similar ultragraph}

\author[H. Larki and N. Rajabzadeh-Hasiri]{Hossein Larki and Najmeh Rajabzadeh-Hasiri}

\address{Department of Mathematics\\
Faculty of Mathematical Sciences and Computer\\
Shahid Chamran University of Ahvaz\\
P.O. Box: 83151-61357\\
Ahvaz\\
 Iran}

\email{h.larki@scu.ac.ir, rajabzadeh.najmeh1396@gmail.com}



\date{\today}


\begin{abstract}
The notion of a self-similar ultragraph $(G,\mathcal{U},\varphi)$ and its $C^*$-algebra $\mathcal{O}_{G,\mathcal{U}}$ were introduced in our recent work, where we proposed inverse semigroup and groupoid models for such $C^*$-algebras as well. In this paper, we investigate minimality and effectiveness of the groupoid of a self-similar ultragraph $(G,\mathcal{U},\varphi)$. In particular, we obtain a result for simplicity of the $C^*$-algebras $\mathcal{O}_{G,\mathcal{U}}$ in a certain case.
\end{abstract}

\maketitle


\section{Introduction}

Ultragraph $C^*$-algebras were introduced in \cite{ana00} to give a graph-like framework for the study of Exel-Laca algebras \cite{li18}. They of course include the graph $C^*$-algebras \cite{rea92}, and many graph-theorical arguments and results in the literature have been extended to the ultragraph setting \cite{ana00,apl00,pin13}. In particular, it is shown that the classes of graph $C^*$-algebras, Exel-Laca algebras, and ultragrah $C^*$-algebras coincide up to Morita-equivalence \cite{li18-boundkry}. Despite these similarities, we sometimes have more difficulty for studying the ultragraph $C^*$-algebras compared with the graph ones. For example, the quotient of an ultragraph $C^*$-algebra by a gauge-invariant ideal is not necessarily of the form of an ultragraph $C^*$-algebra \cite{pin13}, while it is the case for graph $C^*$-algebras. Moreover, the groupoid model for this class of $C^*$-algebras is more complicated (see \cite[Remark 3.7]{bed17}).

Furthermore, in recent years, the algebraic analogue of ultragraph $C^*$-algebras, namely ultragraph Leavitt path algebras, has been considered by many authors \cite{bed17,li1745tuih,li1p0907,dlpo830}.

The notion of a self-similar ultragraph $(G,\mathcal{U}, \varphi )$ and its $C^*$-algebra $\mathcal{O}_{G,\mathcal{U}}$ was introduced in our previous work \cite{ana02}, where inverse semigroup and groupoid models are proposed for such $C^*$-algebras and some properties such as Hausdorffness and $E^*$-unitary of them are described. In this paper, we continue our investigation in this concept, including minimal and effective properties for the groupoid $\mathcal{G}_{G,\mathcal{U}}$ associated to a self-similar ultragraph $(G,\mathcal{U},\varphi)$, and then simplicity of $\mathcal{O}_{G,\mathcal{U}}$ in a certain case. Recall that the minimality and effectiveness are important conditions for a groupoid $\mathcal{G}$, which are necessary, in particular, for simplicity of the universal $C^*$-algebra $C^*(\mathcal{G})$ (cf. \cite{brt92} for example).

Briefly, the paper is organized as follows. After reviewing definitions and background in Section \ref{sec2}, we determine in Sections \ref{sec3} and \ref{sec4} certain conditions for a self-similar ultragraph $(G,\mathcal{U}, \varphi )$, under which the associated tight groupoid $\mathcal{G}_{\mathrm{tight}}(\mathcal{S}_{G,{\mathcal{U} }})$ is minimal and effective. Naturally, these conditions generalize simultaneously the corresponding ones in \cite{bro14} for self-similar graphs and those given in \cite{apl00} for ordinary ultragraphs. Next, in Section \ref{sec5}, we will verify these conditions for some examples of self-similar ultragraphs. In Section \ref{sec6}, we first in Subsection \ref{subs6.1} provide an alternative definition for the $C^*$-algebra $\mathcal{O}_{G,\mathcal{U}}$, which is more analogous to that of a self-similar graph $C^*$-algebra in \cite{bro14}. Then, in Subsection \ref{subs6.2}, we consider a special case: whenever the $1$-cocycle is trivial. In this case, we prove that $\mathcal{O}_{G,{\mathcal{U} }}$ is $*$-isomorphic to the crossed product $C^*(\mathcal{U})\rtimes_\eta G$, where $\eta:G \curvearrowright C^*({\mathcal{U}})$ is the induced action from the action of $G$ on $\mathcal{U}$, and we obtain a result for the simplicity of $\mathcal{O}_{G,{\mathcal{U} }}$.


\section{Preliminaries}\label{sec2}

\subsection{Self-similar ultragraphs}

Ultragraphs generalize directed graphs in the sense that the source of edges could be a set of vertices. So, we write an ultragraph by a quadruple $\mathcal{U}=(U^0,\mathcal{U}^1,r,s)$, where $U^0$ is the vertex set, $\mathcal{U}^1$ is the edge set, and $r,s:\mathcal{U}^1\rightarrow P({U}^0)\setminus\{\varnothing\}$ are the range and source maps such that $r(e)$ is a singleton for every $e \in \mathcal{U}^1$. Then a finite path with the length $n \geq 1$ is a sequence $\alpha=e_1 \ldots e_{n} $ of edges satisfying $r(e_{i+1})\subseteq  s(e_i)$ for all $1\leq i\leq {n-1}$. In this case, we write $|\alpha|=n$ and we denote by $\mathcal{U}^n$ the set of paths of length $n$. Moreover, $r$ and $s$ may be extended on $\mathcal{U}^n$ by setting $r(\alpha)=r(e_1)$ and $s(\alpha)=s(e_n)$. Analogously, $\mathcal{U}^\infty$ stands for infinite paths in $\mathcal{U}$, which are infinite sequences $x= e_1 e_2 \ldots  $ of edges satisfying $r(e_{i+1})\subseteq  s(e_i)$ for all $i \geq 1$.

{\bf Standing assumption.} Throughout the paper, we assume that the ultragraph $\mathcal{U}$ is regular in the sense that every vertex $v\in U^0$ receives only finitely many edges i.e.,
\begin{equation}\label{(2..1)}
0<|\{e\in \mathcal{U}^1 : r(e)=\{v\}\}|< \infty \hspace{5mm} (\forall v\in U^0).
\end{equation}

Given an ultragraph $\mathcal{U}$, we write by $\mathcal{U}^0$ the smallest subset of the power set $P({U}^0)$ containing $\varnothing$, $\{v\}$ for all $v \in{U^0}$, and $s(e)$ for all $e\in{\mathcal{U}^1}$, which is closed under the finite operations $\cap$, $\cup$ and $\backslash$. Then, the (finite) path space of $\mathcal{U}$ is ${\mathcal{U}}^*=\cup_{n=0}^{\infty}{\mathcal{U}}^n$. We will also use the notation $\mathcal{U}^{\geq k}:=\cup_{n=k}^\infty\mathcal{U}^n$ for $k\geq 1$.

\begin{defn}[\cite{ana00}]\label{defn2.3}
Let $\mathcal{U}$ be an ultragraph satisfying (\ref{(2..1)}) for every $v\in U^0$. A {\it Cuntz-Krieger $\mathcal{U}$-family} is a collection of partial isometries $\{s_e:e\in {\mathcal{U}}^{1}\}$ with mutually orthogonal ranges and a collection of projections
 $\{p_A:A\in{\mathcal{U}}^{0}\}$ such that
\begin{enumerate}[(CK1)]
  \item $p_{\emptyset}=0,\;p_{A}p_{B}=p_{A\cap{B}}$ and $p_{A\cup{B}}=p_{A}+p_{B}-p_{A\cap{B}}$ for all $A,B\in {\mathcal{U}}^{0}$,
  \item $s_e^*s_e=p_{s(e)}$ for all $e\in {\mathcal{U}}^1$,
  \item $s_es_e^*\leq p_{r(e)}$ for all $e\in {\mathcal{U}}^1$, and
  \item $p_{\{v\}}=\sum_{r(e)=\{v\}}s_e s_e^*$ for all singletons $\{v\}\in {\mathcal{U}}^0$.
\end{enumerate}
Then the ultragraph $C^*$-algebra $C^*(\mathcal{U})$ is the universal $C^*$-algebra generated by a Cuntz-Krieger $\mathcal{U}$-family. It is shown in \cite{ana00} that such $C^*$-algebra exists with nonzero generators.
\end{defn}

Recall from \cite[ Lemma 2.8]{ana00} that if $A\in{\mathcal{U}}^0$  and $e\in{\mathcal{U}}^1$, then
$$p_As_e=\begin{cases}s_e\;\;\;\;\; \text{if}\;r(e)\subseteq  A\\0\;\;\;\;\;\;\text{otherwise} \end{cases} \text{and}\;\;\; s^*_ep_A=\begin{cases}s^*_e\;\;\;\;\;\text{if}\;r(e)\subseteq  A\\0\;\;\;\;\;\;\;\text{otherwise}. \end{cases}$$
Moreover, for a path $\alpha=e_1 \ldots e_n\in{\mathcal{U}}^*$ we define $s_{\alpha}$ to be $s_{e_1} \ldots s_{e_n}$
if $\left|{\alpha}\right|\geq1$, and  $p_A$ if $\alpha=A\in{\mathcal{U}}^0$. Then
\begin{equation}\label{eq2.2}
C^*(\mathcal{U})=\overline{\mathrm{span}}\{s_\alpha p_A s_\beta^*: \alpha,\beta\in \mathcal{U}^*, A\in \mathcal{U}^0, A\subseteq s(\alpha)\cap s(\beta)\}.
\end{equation}

Now, we recall the definition of a self-similar ultragraph from \cite{ana02}. An {\it{automorphism}} on $\mathcal{U}$ is a bijective map
$$\sigma:{U}^0\sqcup{\mathcal{U}}^1\rightarrow{U}^0\sqcup{\mathcal{U}}^1$$
 such that $\sigma(U^0)\subseteq{U^0}$,
$\sigma({\mathcal{U}}^1)\subseteq{{\mathcal{U}}^1}$, furthermore that $r(\sigma(e))=\sigma(r(e))$ and $s(\sigma(e))=\sigma(s(e))$ for every $e\in{\mathcal{U}}^1 $. Then the collection of all automorphisms on $\mathcal{U}$ forms a group under composition. Let $G$ be a (discrete) group. An {\it{action}} $G\curvearrowright \mathcal{U}$
 is a map $G \times (U^0\sqcup {\mathcal{U}}^1)\rightarrow U^0\sqcup {\mathcal{U}}^1$, denoted by $(g,a)\mapsto g\cdot a$, such that the action for each fixed $g\in G$ gives an automorphism on $\mathcal{U}$. Then, for each $g\in G$ and $A\in {\mathcal{U}}^0$, we define $g\cdot A:=\{ g \cdot v : v\in A\}$. It is shown in \cite[Lemma 3.1]{ana02} that $g\cdot A \in {\mathcal{U}}^0$ provided $A \in {\mathcal{U}}^0$.

\begin{defn}[\cite{ana02}]\label{defn2..7}
Let  $G$ be a discrete group with identity ${{1}_G}$ and $\mathcal{U}$ an ultragraph. A \textit{self-similar ultragraph} is a triple $(G,\mathcal{U},\varphi)$ such that

\begin{enumerate}[(1)]
  \item $\mathcal{U}$ is an ultragraph,
  \item $G$ acts on $\mathcal{U}$ by automorphisms, and
  \item $\varphi:G\times \mathcal{U}^1\rightarrow G$ is a 1-cocycle for $G\curvearrowright \mathcal{U}$ satisfying

\begin{enumerate}
\item $\varphi(gh,e)=\varphi(g,h\cdot e)\varphi(h,e)$\quad (the 1-cocycle property)\;and
 \item $\varphi(g,e)\cdot s(e)  \subseteq g\cdot s(e)$
\end{enumerate}
for all $g,h\in G$ and $e\in \mathcal{U}^1.$
\end{enumerate}
\end{defn}

Observe that we have $\varphi(1_G,e)=1_G$ for all $e\in \mathcal{U}^1$ because
$$\varphi(1_G,e)=\varphi(1_G 1_G,e)=\varphi(1_G,1_G \cdot e)\varphi(1_G,e) =\varphi(1_G,e) \varphi(1_G,e)$$
by the 1-cocycle property. We may inductively extend the action $G\curvearrowright (\mathcal{U}^0 \cup \mathcal{U}^1)$ and the cocycle $\varphi$ on all ${\mathcal{U}}^*$ and also ${\mathcal{U}}^{\infty}$  as follows: for each $\alpha=\alpha_1\alpha_2\in{\mathcal{U}}^*$ and $g\in {G}$, define
$$g\cdot \alpha:=(g\cdot\alpha_1)(\varphi(g,\alpha_1)\cdot \alpha_2)$$
and
$$\varphi(g,\alpha):=\varphi(\varphi(g,\alpha_1),\alpha_2).$$
In the same way, for every  $g\in{G}$ and $x=\alpha_1\alpha_2\cdots \in{\mathcal{U}}^{\infty}$, we may define
$$g\cdot x:=(g\cdot\alpha_1)(\varphi(g,\alpha_1)\cdot\alpha_2)(\varphi(g,\alpha_1\alpha_2)\cdot\alpha_3)\cdots$$
which again belongs to ${\mathcal{U}}^{\infty}$ (see \cite[Proposition 2.5]{bro14}).

Next, we associate a $C^*$-algebra to a self-similar ultragraph $(G,\mathcal{U},\varphi)$.

\begin{defn}[\cite{ana02}]\label{defn2...3}
Given a self-similar ultragraph $(G,\mathcal{U},\varphi)$, a\textit{ $(G,\mathcal{U})$-\textit{family}}  is a collection
$$\{p_{A},s_e:A\in {\mathcal{U}}^0,\; e\in {\mathcal{U}}^1\}\cup \{u_{A,g} :A\in {\mathcal{U}}^0,\;g\in G\}  $$
in a $C^*$-algebra  satisfying the following relations:

\begin{enumerate}[(1)]
  \item $\{p_{A},s_e:A\in {\mathcal{U}}^0,\; e\in {\mathcal{U}}^1\}$ is a Cuntz-Krieger $\mathcal{U}$-family,
  \item $u_{A,{{1}_G}}=p_A$ for all $A\in {\mathcal{U}}^0$,
  \item $(u_{A,g})^*=u_{{{g^{-1}}\cdot A},{g^{-1}}}$ for all $A\in {\mathcal{U}}^0$ and $g\in G$,
  \item $(u_{A,g})(u_{B,h})=u_{{A\cap (g\cdot B)},{gh}}$ for all $A,B\in {\mathcal{U}}^0$ and $g,h\in G$,
  \item $(u_{A,g})s_{e}=\begin{cases}
s_{g\cdot e}u_{{g \cdot s(e)},{\varphi(g,e)}}\;\;\;\;\;\;\text{if }\;g \cdot r(e)\subseteq A \\0\qquad\qquad\qquad\qquad \text{otherwise}
\end{cases} $  for all $A\in {\mathcal{U}}^0, e \in {\mathcal{U}}^1$ and $g\in G$.
\end{enumerate}
Then the $C^*$-algebra ${\mathcal{O}}_{G,\mathcal{U}}$ associated to $(G,\mathcal{U},\varphi)$ is the universal $C^*$-algebra generated by a $(G,\mathcal{U})$-family $\{p_A,s_e,u_{A,g}\}$.
\end{defn}

Note that, combining relations (2), (3) and (5) above, one may obtain by induction that
$$(u_{A,g})s_{\alpha}=
\begin{cases}
s_{g\cdot \alpha}u_{{g \cdot s(\alpha)},{\varphi(g,\alpha)}}\;\;\;\;\;\;\text{if }\;g \cdot r(\alpha)\subseteq A \\0\qquad\qquad\qquad\qquad \text{otherwise}
\end{cases}
$$
for all $A\in {\mathcal{U}}^0, \alpha \in {\mathcal{U}}^*$ and $g\in G$, where $s_\alpha:=p_A$ if $\alpha=A\in \mathcal{U}^0$ and $s_\alpha:=s_{e_1}\ldots s_{e_n}$ if $\alpha=e_1\ldots e_n\in \mathcal{U}^n$ for $n\geq 1$. According to \cite[Corollary 8.5]{ana02}, for any self-similar $(G,\mathcal{U},\varphi)$, the $C^*$-algebra ${\mathcal{O}}_{G,\mathcal{U}}$ with nonzero generators $\{p_A,s_e,u_{A,g}\}$ exists and we have
\begin{equation}\label{(2.1)}
{\mathcal{O}}_{G,\mathcal{U}}=\mathrm{\overline{span}}\{s_{\alpha}u_{A,g} s^*_{\beta}\;:g\in G,\; \alpha,\beta \in{\mathcal{U}}^*, A\in {\mathcal{U}}^0 ~ \text{and} ~ A\subseteq{s(\alpha)\cap g\cdot s(\beta})\}.
\end{equation}

\begin{rem}
In Subsection \ref{subs6.1} below, we will give an alternative definition for the $C^*$-algebra $\mathcal{O}_{G,\mathcal{U}}$ which is more analogous to that of $\mathcal{O}_{G,E}$ in \cite{bro14} for self-similar graphs. In fact, in Proposition \ref{prop6.2} we show, for each $g\in G$, that the series $\sum_{v\in U^0} u_{\{v\},g} $ converges to an element $u_g$ in the multiplier algebra $M(\mathcal{O}_{G,\mathcal{U}})$ such that the map $G\rightarrow M(\mathcal{O}_{G,\mathcal{U}})$, $g \mapsto u_g$,  is a unitary $*$-representation of $G$.
\end{rem}

\subsection{Groupoids and inverse semigroups}

In order to set notations, we briefly recall necessary definitions and background about groupoids, inverse semigroups, and groupoids of germs. The reader can refer to \cite{exe21,exe17,exe08} for more details.

A {\it groupoid} is a small category $\mathcal{G}$ in which every morphism has an inverse. The set of objects is called the {\it unit space} of  $\mathcal{G}$, which can be identified with $\mathcal{G}^{(0)}:=\{\lambda^{-1}\lambda:\lambda\in \mathcal{G}\}$. A multiplication $\lambda \eta$ in $\mathcal{G}$ is well-defined only if $s(\lambda )=r(\eta)$, where $s(\lambda ):=\lambda^{-1}\lambda$ and $r(\eta):=\eta \eta ^{-1}$. In this article, we will deal with topological groupoids, which are equipped with a topology such that multiplication, inversion, and the maps $r,s:\mathcal{G}  \mapsto \mathcal{G}^{(0)}$ are all continuous.

An \emph{inverse semigroup} is a discrete semigroup ${S}$ with an inverse map $s \mapsto   s^*$ in the sense that for every $s\in {S}$, there exists a unique $s^*\in S$ satisfying
$$s=s s^* s \hspace{5mm} \mathrm{and} \hspace{5mm} s^*=s^*s s^*.$$
Let ${S}$ be an inverse semigroup with $0$; the zero satisfies $0 s=s0=0$ for all $s \in {S}$. We denote by $\mathcal{E}( {S})$ the idempotent set in $ {S}$. It follows that $\mathcal{E}({S})$ is a commutative semilattice with the meet $e \wedge f:=ef$, and moreover $e^*=e$ for all $e \in \mathcal{E}({S})$. For $e, f \in \mathcal{E}( {S})$, we say $e$ \emph{intersects} $f$, denoted by $e \Cap f$, if $ef \neq 0$. We will often consider the partial order $\leq$ on $ {S}$ defined by $s \leq  t \Longleftrightarrow     s=te$ for some $e \in \mathcal{E}( {S})$; in particular, in the case $s,t \in \mathcal{E}( {S})$, we have $s \leq t$ if and only if $st=s$.

For every $e \in \mathcal{E}({S})$ let us use the notation $e\uparrow$ for the set
$$e\uparrow:=\{f\in \mathcal{E}( {S}):e\leq f\}$$
in $\mathcal{E}({S})$. A{\it{ filter}} in $\mathcal{E}({S})$ is a nonempty subset $\mathcal{F} \subseteq \mathcal{E}({S})$ satisfying
\begin{enumerate}[(1)]
  \item $e,f \in \mathcal{F} \Longrightarrow ef \in \mathcal{F}$, and
  \item $e\in \mathcal{F} \Longrightarrow e\uparrow \subseteq \mathcal{F}$.
\end{enumerate}
An {\it ultrafilter} is a proper maximal filter in $\mathcal{E}( {S})$. The set of all filters (ultrafilters) without 0 in $\mathcal{E}( {S})$ is denoted by $\widehat{\mathcal{E}}_0(S)$ ($\widehat{\mathcal{E}}_\infty(S)$ respectively). Then $\widehat{\mathcal{E}}_0(S)$ is a topological space with the neighborhoods
$$N(e;e_1,\ldots ,e_n):=\{\mathcal{F}\in \widehat{\mathcal{E}}_0(S): e\in \mathcal{F}, e_i \notin \mathcal{F}, \; 1\leq i \leq n\} \hspace{5mm} (e,e_i\in \mathcal{E}(S)).$$

\begin{defn}
The {\it  tight filter space} of ${S}$ is the closure of $\widehat{\mathcal{E}}_\infty({S})$ in $\widehat{\mathcal{E}_0}(S)$, which is denoted by $\widehat{\mathcal{E}}_{\mathrm{tight}}({S})$. Then, each ${\mathcal{F}}\in \widehat{\mathcal{E}}_{\mathrm{tight}}({S})$ is called a {\it tight} filter.
\end{defn}

Following \cite{exe21}, there is a natural action of $S$ on the tight filter space $\widehat{\mathcal{E}}_{\mathrm{tight}}(S)$. To define it, write
$$D^e:=\{ \mathcal{F} \in \widehat{\mathcal{E}}_{\mathrm{tight}}(S)  : e\in  \mathcal{F}  \}$$
for every $e\in {\mathcal{E}}(S)$. Then the action $S\curvearrowright \widehat{\mathcal{E}}_{\mathrm{tight}}(S)$ is $s\mapsto \theta_s$, where $\theta_s:D^{s^* s}\rightarrow D^{s s^*}$ maps every $\mathcal{F}\in D^{s^* s}$ to $(s\mathcal{F}s^*)\uparrow\in D^{s s^*}$. Now consider the set
$$S*\widehat{\mathcal{E}}_{\mathrm{tight}}(S)=\{(s,\mathcal{F})\in S\times \widehat{\mathcal{E}}_{\mathrm{tight}}(S): s^*s\in \mathcal{F}\}$$
and define the relation $(s,\mathcal{F})\sim (s',\mathcal{F}')$ whenever $\mathcal{F}=\mathcal{F}'$ and $se=s'e$
for some $e\in \mathcal{F}$. Let $[s,\mathcal{F}]$ be the  equivalence class of $(s,\mathcal{F})$. Then the {\it tight groupoid} associated to $S$ is $\mathcal{G}_{\mathrm{tight}}(S)=S*\widehat{\mathcal{E}}_{\mathrm{tight}}(S)/\sim$ with the multiplication
$$[t,\theta_s(\mathcal{F})][s,\mathcal{F}]:=[ts,\mathcal{F}]$$
and inversion
$$[s,\mathcal{F}]^{-1}:=[s^*,\theta_s(\mathcal{F})].$$
In addition, the source and  range maps are
$$ s([s,\mathcal{F}])=[s^*s,\mathcal{F}]  \hspace{5mm} \mathrm{and} \hspace{5mm}  r([s,\mathcal{F}])=[ss^*,\theta_s(\mathcal{F})],$$
and we may identify the unit space of $\mathcal{G}_{\mathrm{tight}}(S)$ with $\widehat{\mathcal{E}}_{\mathrm{tight}}(S)$ via the correspondence $[e,\mathcal{F}]\mapsto \mathcal{F}$.

\subsection{ The inverse semigroup associated to $(G,\mathcal{U},\varphi)$ }\label{sub2.4}
Let us here review from \cite{ana02} the inverse semigroup and tight groupoid associated to a self-similar ultragraph $(G,\mathcal{U},\varphi)$. We set $\mathcal{U}^{\sharp}:=\{\omega\} \cup \mathcal{U}^{\geq 1}$, where $\omega$ is denoted for the universal path of length zero in the sense that $\omega \alpha=\alpha \omega=\alpha$ for all $\alpha \in {\mathcal{U}}^{\geq 1}$. If necessary, we will write $s(\omega)=r(\omega)=U^0$ in computations and moreover, define $g \cdot \omega:=\omega$ and $\varphi(g,\omega):=g$ for every $g\in G$. Then the inverse semigroup ${\mathcal{S}}_{G,\mathcal{U}}$ is
\begin{equation}\label{(5...1)}
{\mathcal{S}}_{G,\mathcal{U}}: =\{(\alpha,A,g,\beta): \alpha,\beta \in \mathcal{U}^{\sharp}, ~ A\in {\mathcal{U}}^0, ~ g\in G  , \;\emptyset\neq A\subseteq{s(\alpha)\cap g \cdot s(\beta})\}\cup \{0\},
\end{equation}
with the multiplication
\begin{align*}
&(\alpha, A,g,\beta)(\gamma,B,h,\delta)=\\
&\left\{
  \begin{array}{ll}
    \big(\alpha(g\cdot \varepsilon),(g\cdot s(\varepsilon))\cap (\varphi(g,\varepsilon)\cdot B),\varphi(g,\varepsilon)h,\delta \big) & \hspace{-10mm} \text{if} ~ \gamma=\beta\varepsilon, ~ |\varepsilon|\geq 1, ~ g \cdot r(\varepsilon)\subseteq A, \\
    \big(\alpha, A \cap (g{\varphi(h^{-1},\varepsilon)}^{-1} h^{-1}\cdot s(\beta)),g{\varphi(h^{-1},\varepsilon)}^{-1},\delta (h^{-1}\cdot\varepsilon) \big) & \\
     & \hspace{-10mm} \text{if} ~ \beta=\gamma\varepsilon, ~ |\varepsilon|\geq 1, ~ r(\varepsilon)\subseteq B, \\
    \big(\alpha,A\cap (g\cdot  B),g h,\delta \big) & \hspace{-10mm} \text{if} ~ \gamma=\beta, ~ A\cap (g \cdot B) \neq \emptyset, \\
    ~ ~ 0 & \hspace{-10mm} \text{otherwise},
  \end{array}
\right.
\end{align*}
and the inverse
$$(\alpha,A,g,\beta)^*:=(\beta,g^{-1}\cdot A,g^{-1},\alpha).$$
It is shown in \cite[Proposition 5.3]{ana02} that ${\mathcal{S}}_{G,\mathcal{U}}$ is an inverse semigroup with the idempotent set
$$\mathcal{E}({\mathcal{S}}_{G,\mathcal{U}})=\{ q_{(\alpha,A)}:\; \alpha\in \mathcal{U}^{\sharp} , A\subseteq s(\alpha) \}$$
where $ q_{(\alpha,A)}:=(\alpha,A,1_G,\alpha)$. In particular,

\begin{cor}[{\cite[Corollary 5.6]{ana02}}]\label{cor2.5}
Let $(G,\mathcal{U},\varphi)$ be a self-similar ultragraph. For $q_{(\alpha,A)},q_{(\beta,B)}\in\mathcal{E}({\mathcal{S}_{G,\mathcal{U} }})$, we have $q_{(\alpha,A)}\le q_{(\beta,B)}$ if and only if one of the following holds:
\begin{enumerate}[$(1)$]
 \item $\alpha=\beta$ and $A\subseteq B$, or
 \item $\alpha=\beta\gamma$ for some $\gamma \in {\mathcal{U}}^{\geq 1}$ with $r(\gamma)\subseteq B$.
\end{enumerate}
\end{cor}

Note that the idempotent set $\mathcal{E}({\mathcal{S}}_{G,\mathcal{U}})$ coincides with the one $\mathcal{E}({\mathcal{S}}_{\mathcal{U}})$ for the inverse semigroup ${\mathcal{S}}_{\mathcal{U}}$ of the ultragraph $\mathcal{U}$ {\cite{bed17}}. So, they have same filter spaces, i.e.
${\widehat{\mathcal{E}}}_{0}({\mathcal{S}}_{G,\mathcal{U}})={\widehat{\mathcal{E}}}_{0}({\mathcal{S}}_{\mathcal{U}})$, and we may apply the description of the tight filter space ${\widehat{\mathcal{E}}}_{\mathrm{tight}}({\mathcal{S}}_{\mathcal{U}})$ in \cite[Section 3.1]{bed17} for $\widehat{\mathcal{E}}_{\mathrm{tight}}(\mathcal{S}_{G,\mathcal{U}})$. (It is originally proved in \cite{boa17} for labelled graphs and its ultragraph description in \cite{bed17} is deduced from Theorems 5.10 and 6.7 of \cite{boa17}.) To ease the notation, we will denote by $\mathcal{T}$ the tight filter space ${\widehat{\mathcal{E}}}_{\mathrm{tight}}({\mathcal{S}}_{G,\mathcal{U}})={\widehat{\mathcal{E}}}_{\mathrm{tight}}({\mathcal{S}}_{\mathcal{U}})$ in the sequel.

\begin{prop}[{\cite{boa17,bed17}}]\label{prop2.11}
Every tight filter $\mathcal{F}$ in $\mathcal{T}$ can be uniquely described as one of the following forms:
\begin{enumerate}[$(1)$]
  \item $\mathcal{F}$ is associated to a pair $(\alpha,\mathcal{B})$, where $\alpha\in \mathcal{U}^{\sharp}$ and $\mathcal{B}$ is a filter in the set
$${\mathcal{B}}_{s(\alpha)}:=\{A \in {{\mathcal{U} }^0} : A\subseteq s(\alpha)\}$$
such that $|A|=\infty$ for all $A \in \mathcal{B}$. If $\alpha=\alpha_1 \ldots \alpha_n$, then
$$\mathcal{F}={\mathcal{F}}_{(\alpha,\mathcal{B})}:=\{q_{(\alpha,A)}: A \in \mathcal{B}\} \cup \{q_{(\beta,A)} : |\beta|<|\alpha|, \beta={\alpha}_1\ldots \alpha_{|\beta|} ~ \text{and} ~ r(\alpha_{{|\beta|}+1})\subseteq A\}.$$
  \item $\mathcal{F}$ is associated to an infinite path $x=\alpha_1\alpha_2 \ldots \in {\mathcal{U}}^{\infty}$ such that
$$\mathcal{F}={\mathcal{F}}_x:=\{q_{(\beta,A)}:\: \beta={\alpha}_1\alpha_2\ldots \alpha_{|\beta|}\; \text{and} \;r(\alpha_{{|\beta|}+1})\subseteq A \subseteq s(\alpha_{|\beta|})\}.$$
\end{enumerate}
\end{prop}

Next, one may construct the tight groupoid ${{\mathcal{G}}}_{\mathrm{tight}}({\mathcal{S}}_{G,\mathcal{U}})$ in a natural way. Given any $s=(\alpha,A,g,\beta)$ in ${\mathcal{S}}_{G,\mathcal{U}}$, we have $ss^*=q_{(\alpha,A)}$ and $s^*s=q_{(\beta,g^{-1}\cdot A)}$, so
$$D^{ss^*}=\{ \mathcal{F} \in \mathcal{T}: \;  q_{(\alpha,A)}\in \mathcal{F} \}$$
which is denoted by $\mathrm{Z} (\alpha,A)$, and similarly $D^{s^*s}=\mathrm{Z} (\beta,g^{-1}\cdot A)$. Note that the sets $\mathrm{Z} (\alpha,A)$, where $\alpha\in   {\mathcal{U}}^{\sharp}$ and $A\subseteq s(\alpha)$, are compact open sets generating the topology on $\mathcal{T}$.$\;$Thus, the tight groupoid of ${\mathcal{S}}_{G,\mathcal{U}}$ is
$${\mathcal{G}}_{\mathrm{tight}}({\mathcal{S}}_{G,\mathcal{U}})=\{  [(\alpha,A,g,\beta),\mathcal{F} ]: \; \mathcal{F}\in  \mathrm{Z} (\beta,g^{-1}\cdot A)\},$$
which is an ample groupoid with the Hausdorff unit space ${{\mathcal{G}}}_{\mathrm{tight}}^{(0)}({\mathcal{S}}_{G,\mathcal{U}})=\mathcal{T}$.


\section{Minimality of $\mathcal{G}_{\mathrm{tight}}(\mathcal{S}_{G,{\mathcal{U} }})$}\label{sec3}

In this section, we provide certain conditions for a self-similar ultragraph, under which the associated tight groupoid is minimal. Our conditions extend those in \cite{ana00} for an ordinary ultragraph. Recall that, for a groupoid $\mathcal{G}$, a subset $U\subseteq \mathcal{G}^{(0)}$ is called \emph{invariant} if for every $\alpha\in \mathcal{G}$ we have $s(\alpha)\in U ~~\Longleftrightarrow ~~ r(\alpha)\in U$. Then, we say $\mathcal{G}$ is a \emph{minimal groupoid} if the only invariant open subsets of $\mathcal{G}^{(0)}$ are the empty set and $\mathcal{G}^{(0)}$ itself.

\begin{defn}\label{defn3..1}
We say $(G,\mathcal{U},\varphi)$ is {\it{G-cofinal}} if it satisfies the following conditions:
\begin{enumerate}[(1)]
\item For each infinite path $x=x_1x_2 \ldots$ and $v \in U^0$, there is a pair $(g,\gamma) \in  G\times\mathcal{U} ^{\geq 1}$  such that $r(\gamma)=\{g \cdot v\}$ and $s(\gamma)=s(x_n)$ for some $n \geq 1$.
    
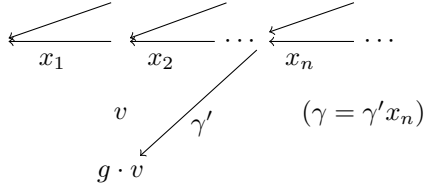
\begin{center}
\begin{tikzpicture}[scale=0.8]
\node (1) at (-13,0) {};
\node (2) at (-11,0) {};
\node (3) at (-9,0) {$\cdots$};
\node (10) at (-9.1,0.7) {};
\node (9) at (-8.6,0) {};
\node (4) at (-7,0) {};
\node (5) at (-6.7,0) {$\cdots$};

\node (gv) at (-11,-2.2) {$g \cdot v$};
\node (6) at (-11,0.7) {};
\node (7) at (-9,0.7) {};
\node (8) at (-7,0.7) {};
\node (gamma) at (-7,-1.2) {$(\gamma={\gamma}' x_n)$};
\node (v) at (-11,-1.2) {$v$};

\path[->] (2) edge node[above=-14pt] {$\hspace{-2mm}x_1$} (1);
\path[->] (3) edge node[above=-14pt] {$\hspace{-3mm}x_2$} (2);
\path[->] (4) edge node[above=-14pt] {$\hspace{-3mm}x_n$} (3);
\path[->] (6) edge node[above=-14pt] {$\hspace{-3mm}$} (1);
\path[->] (10) edge node[above=-14pt] {$\hspace{-3mm}$} (2);
\path[->] (8) edge node[above=-14pt] {$\hspace{-3mm}$} (3);
\path[->] (9) edge node[above=-16pt] {$\hspace{1mm}\gamma'$} (gv);

\end{tikzpicture}
\end{center}
\item For every $e \in \mathcal{U} ^1$ with $|s(e)|=\infty$ and each $v \in U^0$, there are finitely many pairs $\{(g_i,\gamma_i)\}_{i=1}^{n}\subseteq G \times {\mathcal{U} }^{\geq 1}$ such that $r(\gamma_i)=\{g_i \cdot v\}$ and $s(e)\subseteq \bigcup_{i=1}^n s(\gamma_i)$.
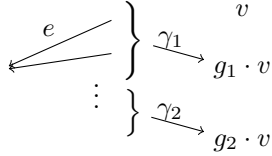
    \begin{center}
\begin{tikzpicture}[scale=0.8];
\node (set) at (-5.8,1) {$\Bigg\}$};
\node (2) at (-6,1.4) {};
\node (3) at (-6,0.8) {};
\node (setss) at (-5.8,-0.2) {$ \Big\}$};
\node (4) at (-6.4,0.2) {$\vdots$};
\node (v) at (-4,1.5) {$v$};
\node (g1v) at (-4,0.5) {$g_1\cdot v$};
\node (g2v) at (-4,-0.7) {$g_2\cdot v$};
\node (1) at (-8,0.5) {};

\path[->] (set) edge node[above=-2pt] {$\hspace{-2mm}{\gamma_1}$} (g1v);
\path[->] (setss) edge node[above=-2pt] {$\hspace{-2mm}{\gamma_2}$} (g2v);

\path[<-] (1) edge node[above=6pt] {$\hspace{-3mm}e$} (3);
\path[<-] (1) edge node[above=-2pt] {$\hspace{-2mm}$} (2);

\end{tikzpicture}

\end{center}
\end{enumerate}
\end{defn}

In \cite{kat08}, Exel and Pardo proposed a condition for an inverse semigroup $S$, which is equivalent to the minimality of the tight groupoid ${\mathcal{G} }_{\text{tight}}({S})$ \cite[Theorem 5.5]{kat08}. We state this result in Proposition \ref{prop3.3} below for convenience. Before that, let us recall the notion of (outer) cover for an element $s\in S$.

\begin{defn}
Let ${S}$ be an inverse semigroup with zero, and $s\neq0$ an element of ${S}$. A subset $\mathcal{C}\subseteq \mathcal{E}({S})$ is called an {\emph{outer cover}} for $s$  if for any idempotent $e\leq s$, there exists $f \in\mathcal{C}$ satisfying $ef\neq 0$ (in this case we write $e\Cap f)$. A {\it{cover}} for $s$ is an outer cover $\mathcal{C}$ that is contained in $\{e\in \mathcal{E}({S}): e\leq s\}$.
\end{defn}

\begin{prop}[{\cite [Theorem 5.5]{kat08}}]\label{prop3.3}
Let ${S}$ be an inverse semigroup with zero. Then the following are equivalent
\begin{enumerate}[$(1)$]
\item The standard action $S\curvearrowright \widehat{\mathcal{E}}_{\mathrm{tight}}$ is irreducible,
\item $\mathcal{G}_{\mathrm{tight}}({S})$ is minimal,
\item for every nonzero $e$ and $f$ in ${\mathcal{E}}$, there are $s_1,\ldots,s_n$ in $S$ such that $\{s_if{s^{*}_i}\}_{1\leq i \leq n}$ is an outer cover for $e$.
\end{enumerate}
\end{prop}

In order to prove the main result of this section, we will also need the following lemma.

\begin{lem}\label{lem9.2}
Let $(G,\mathcal{U} ,\varphi)$ be $G$-cofinal. Then for every singleton set $\{v\}, \{v'\} \in {\mathcal{U} }^0$, there are finitely many elements $s_1,\ldots, s_n \in \mathcal{S}_{G,{\mathcal{U} }}$ such that $\{s_i q_{(\omega,\{v\})} s_i^*\}^n_{i=1}$ is an outer cover for $q_{(\omega,\{v'\})}$.
\end{lem}

\begin{proof}
Fix $\{v\},\{v'\}\in \mathcal{U}^0$. We focus on the compact open cylinder $\mathrm{Z}(\omega,\{v'\})$, which contains ultrafilters of the forms ${\mathcal{F} }_x$ and ${\mathcal{F} }_{(\alpha,\mathcal B)}$ by Proposition \ref{prop2.11}. First, assume ${\mathcal{F} }_x \in \mathrm{Z}(\omega,\{v'\})$. Then $r(x)=\{v'\}$ and the $G$-cofinality says that there are a subpath $\alpha_x \in {\mathcal{U} }^{\geq 1}$ with $x=\alpha_x y$ and $(g,\gamma)\in G\times{\mathcal{U} }^{\geq 1}$ such that $r(\gamma)=\{g \cdot v\}$ and $s(\gamma)=s(\alpha_x )$. In particular, we have ${\mathcal{F} }_x \in \mathrm{Z}(\alpha_x ,s(\alpha_x ))$. Second, for every ultrafilter of the form ${\mathcal{F} }_{(\alpha,\mathcal B)}$ in $\mathrm{Z}(\omega,\{v'\})$, it is clear that ${\mathcal{F} }_{(\alpha,\mathcal B)}\in \mathrm{Z}(\alpha ,s(\alpha ))$. Therefore,
$$\mathrm{Z}(\omega,\{v'\})\subseteq
 \bigcup_{\substack{x\in \mathcal{U}^\infty\\ r(x)= \{v'\} }} {\mathrm{Z} \big(\alpha_x,s(\alpha_x)\big)}\ \cup \bigcup_{\substack{\alpha\in \mathcal{U}^{\geq 1}\\ r(\alpha)= \{v'\}\\ |s(\alpha)|=\infty }} {\mathrm{Z} \big(\alpha,s(\alpha)}\big).$$
Since $\mathrm{Z}(\omega,\{v'\})$ is compact, there are finite subsets
$$X \subseteq\{\alpha_x : x\in \mathcal{U}^\infty,\; r(x)=\{v'\}\}$$
 and
$$Y\subseteq\{\alpha \in \mathcal{U}^{\geq 1} : r(\alpha)=\{v'\}\;\text{and}\;  |s(\alpha)|=\infty\}$$
such that
\begin{equation}\label{(9.1)}
\mathrm{Z}(\omega,\{v'\})\subseteq \bigcup_{\alpha \in X \cup Y} {\mathrm{Z} \big(\alpha,s(\alpha)\big)}.
\end{equation}

Now we try to find an outer cover of conjugates of $q_{(\omega,\{v\})}$ for $q_{(\omega,\{v'\})}$. As said in the begining of the proof, for each $\alpha \in X$, there is $(g,\gamma)\in G \times {\mathcal{U} }^{\geq 1}$ such that $r(\gamma)=\{g \cdot v\}$ and $s(\gamma)=s(\alpha)$. Thus if we define $s_\alpha:=(\alpha,s(\alpha),g,g^{-1}\cdot \gamma) \in \mathcal{S}_{G,{\mathcal{U} }}$, then

\begin{align}\label{(9.2)}
s_\alpha q_{(\omega,\{v\})} s_{\alpha}^*&=(\alpha,s(\alpha),g,g^{-1}\cdot \gamma) (\omega,\{v\},1_G,\omega)(g^{-1}\cdot \gamma,g^{-1}\cdot s(\alpha),g^{-1},\alpha)\nonumber\\
&=(\alpha,s(\alpha),g,g^{-1}\cdot \gamma)(g^{-1}\cdot \gamma,g^{-1}\cdot s(\alpha),g^{-1},\alpha)\;\qquad (\text{since} \;s(\gamma)=s(\alpha)) \nonumber\\
&=(\alpha,s(\alpha),1_G,\alpha)\nonumber\\
&= q_{(\alpha,s(\alpha))}.
\end{align}
Furthermore, for every $\alpha \in Y$, since $|s(\alpha)|=\infty$, the $G$-cofinality implies that there are $(g_{\alpha,1},\gamma_{\alpha,1}),\ldots,(g_{\alpha,n_\alpha},\gamma_{\alpha,n_\alpha})\in G\times {\mathcal{U} }^{\geq 1} $ such that
\begin{equation}\label{(9.3)}
{\{g_{\alpha,i}\cdot v\}= r(\gamma_{\alpha,i})\;\text{and} \;s(\alpha)\subseteq \bigcup_{i=1}^{n_\alpha}s(\gamma_{\alpha,i})}.
\end{equation}
For convenience, write $B_{\alpha,i}:=s(\alpha)\cap s(\gamma_{\alpha,i})$ for $1\leq i \leq n_\alpha$. If we define $s_{\alpha,i}:=(\alpha,B_{\alpha,i},g_{\alpha,i},{g^{-1}_{\alpha,i}}\cdot \gamma_{\alpha,i})$, then
\begin{equation}\label{(9.4)}
{s_{\alpha,i}q_{(\omega,\{v\})} {s^*}_{\alpha,i}=q_{(\alpha,B_{\alpha,i})}}
\end{equation}
for all $1\leq i \leq n_\alpha$.

We will show that the finite collection
\begin{equation}\label{(9.5)}
{\{s_\alpha q_{(\omega,\{v\})}s_{\alpha}^* : \alpha \in X\}\cup (\bigcup_{\alpha  \in Y}\{s_{\alpha,i}q_{(\omega,\{v\})} s^*_{\alpha,i} : 1\leq i \leq n_\alpha\})}\end{equation}
is an outer cover for $q_{(\omega,\{v'\})}$. First, a claim:

\textbf{Claim.}
For each $q_{(\tau,C)}\leq q_{(\omega,\{v'\})}$, there are an extension $\beta=\tau\eta $ of $\tau$ with $r(\eta)\subseteq  C$ and $\alpha_0 \in X \cup Y$ such that $\alpha_0$ is a subpath of $\beta$, i.e. $\tau \eta= \beta=\alpha_0\alpha$.

{\it Proof of Claim.} It follows from $(\ref{(9.1)})$ and Corollary \ref{cor2.5}. Indeed, since $\mathcal{U}$ contains no sources, we may pick $y=y_1y_2 \dots \in {\mathcal{U}}^\infty$ with $r(y)\subseteq  C$, and hence $\tau y\in {\mathcal{U} }^\infty$. If we choose an integer $n\geq \max \{|\alpha|\;:\; \alpha \in X \cup Y\}$ and define $\beta:= \tau  y_1 \cdots y_n$, then
 $$|\beta|\geq \max _{\alpha \in X \cup Y}|\alpha|$$
and we have $q_{(\beta,s(\beta))}\leq q_{(\tau, C)}$. Moreover, $ \mathcal{F}_{\tau y}\in \mathrm{Z}(\omega,\{v'\})$ and by $(\ref{(9.1)})$, there is $\alpha_0 \in X \cup Y$ such that $ \mathcal{F}_{\tau y}\in \mathrm{Z}(\alpha_0,s(\alpha_0))$. In this case, we must have $\tau y=\alpha_0  z$ for some $z\in {\mathcal{U} }^\infty$, and since $\tau y=\beta y_{n+1}y_{n+2} \ldots$ with $|\beta|>|\alpha_0|$, we deduce that $\alpha_0$ is a subpath of $\beta$ of the form $\beta=\alpha_0\alpha^{'}$ and Claim is proved.

In order to continue the proof, pick a sub-idempotent $q_{(\beta,{C})}\leq q_{(\omega,\{v'\})}$. By the above Claim, there is a common extension $\beta'$ of $\beta$ and some $\alpha \in X \cup Y$ such that $q_{({\beta}^{'},s({\beta}^{'}))}\leq q_{(\beta,{C})}$. So, without loss of generality, we may replace $\beta$ by ${\beta}^{'}$ and assume $\beta=\alpha{\alpha}^{'}$ for some $\alpha \in X \cup Y$. In the case $\alpha \in X$, since $s_\alpha q_{(\omega,\{v\})}s_{\alpha}^*=q_{(\alpha,s(\alpha))}$ by (\ref{(9.2)}), and $q_{(\beta,s(\beta))}q_{(\alpha,s(\alpha))}=q_{(\beta,s(\beta))}$, we have
$$q_{(\beta,s(\beta))}\Cap s_\alpha q_{(\omega,\{v\})}s_{\alpha}^*.$$
In the other case, if ${\beta}=\alpha{\alpha}^{'}$ with $\alpha \in Y$, then $r({\alpha}^{'})\subseteq s({\alpha})$. Recall from (\ref{(9.3)}) that $s(\alpha)\subseteq \bigcup_{i=1}^{n_\alpha}s(\gamma_{\alpha,i})$, so
$$\bigcup_{i=1}^{n_\alpha}B_{\alpha,i}=\bigcup_{i=1}^{n_\alpha}(s({\alpha}) \cap s(\gamma_{\alpha,i}))=s({\alpha}) \cap (\bigcup_{i=1}^{n_\alpha}(s(\gamma_{\alpha,i}))=s({\alpha}).$$
Hence $r({\alpha}^{'})\subseteq \bigcup_{i=1}^{n_\alpha}B_{\alpha,i}$ and there is $1\leq j \leq {n_\alpha} $ such that $r({\alpha}^{'}) \subseteq B_{\alpha,j}$ (because $r({\alpha}^{'})$ is a singleton). This follows that $q_{(\beta,s(\beta))} \leq q_{(\alpha,B_{\alpha,j})} $ and therefore
$$q_{(\beta,s(\beta))}\Cap s_{\alpha_j}q_{(\omega,\{v\})} s_{\alpha_j}^*.$$
Consequently, the collection (\ref{(9.5)}) is a finite outer cover for $q_{(\omega,\{v'\})}$, completing the proof.
\end{proof}

\begin{thm}\label{thm3.4}
If $(G,\mathcal{U},\varphi)$ is a $G$-cofinal self-similar ultragraph, then the groupoid $\mathcal{G}_{\mathrm{tight}}(\mathcal{S}_{G,{\mathcal{U} }})$ is minimal.
\end{thm}

\begin{proof}
In light of \cite[Theorem 5.5]{kat08}, it suffices to show that, for every $q_{(\alpha,A)},q_{(\beta,B)} \in \mathcal{E}(\mathcal{S}_{G,{\mathcal{U} }})$, there exist $s_1,\ldots,s_n \in \mathcal{S}_{G,{\mathcal{U} }}$ such that $\{s_i q_{(\beta,B)} s_{i}^*\}_{i=1}^{n}$ is an outer cover for $q_{(\alpha,A)}$. Since we have $s q_{(\beta,B)} s^*=q_{(\omega,B)} $ with $s:=(\omega,B,1_G,\beta)$ and also $q_{(\alpha,A)} \leq q_{(\omega,r(\alpha))}$, we can assume without loss of generality that $q_{(\alpha,A)}=q_{(\omega,A)}$ and $q_{(\beta,B)}=q_{(\omega,B)}$.

So fix $A,B\in \mathcal{U}^0$ and pick some $v\in B$. Then, by \cite[Lemma 2.12]{ana00} there is a finite set $V \subseteq U^0$ of vertices and a finite set $E$ of edges $e\in \mathcal{U}^1$ with $|s(e)|=\infty$ such that
\begin{equation}\label{(9.6)}
{A\subseteq (\bigcup_{e\in E}s(e))\cup V}.
\end{equation}
We want to find a finite outer cover for $\{q_{(\omega,\{w\})},q_{(\omega,s(e))}: w\in V ,e\in E\}$ by conjugates of $q_{(\omega,\{v\})}$. For each $w\in V$, Lemma \ref{lem9.2} says that there is a finite set $M_w  \subseteq \mathcal{S}_{G,{\mathcal{U} }}$ such that $\{sq_{(\omega,\{v\})}s^*: s\in M_w\}$ is an outer cover for $q_{(\omega,\{w\})}$. Moreover, for each $e\in E$, since $|s(e)|=\infty$, condition (2) of the $G$-cofinality gives finitely many pairs $(g_{1}^e,\gamma_{1}^e),\ldots, (g_{n_e}^e,\gamma_{n_e}^e)\in G \times {\mathcal{U} }^{\geq 1}$ satisfying
$$r(\gamma_{i}^e)=\{g_{i}^e \cdot v \} \qquad (\forall \; 1\leq i \leq n_e)$$
and $s(e)\subseteq \bigcup_{i=1}^{n_e}s(\gamma_{i}^e)$. Denote $A_{i}^e:=s(e) \cap s(\gamma_{i}^e)$ for $1\leq i \leq n_e$. If $s_{i}^e:=(\omega,A_{i}^e,g_{i}^e,(g_{i}^e)^{-1}\cdot \gamma_{i}^e)$, then we have
$$s_{i}^e q_{(\omega,\{v\})} {(s_{i}^e)}^*=q_{(\omega,A_{i}^e)} \qquad (\forall \; 1\leq i \leq n_e).$$
Since $\bigcup_{i=1}^{n_e}A_{i}^e= s(e)$, it follows that $\{s_i^e q_{(\omega,\{v\})}(s_i^e)^*: 1\leq i\leq n_e\}$ is an outer cover for $q_{(\omega,s(e))}$. Now set $M_e:=\{s_{i}^e : 1\leq i \leq n_e\}$ for every $e\in E$. Then $M=(\bigcup_{w \in V}M_w)\cup (\bigcup_{e\in E}M_e)$ is a finite subset of $\mathcal{S}_{G,{\mathcal{U} }}$ such that $\{sq_{(\omega,\{v\})} s^* : s\in M\}$ is an outer cover for $\{q_{(\omega,\{w\})},q_{(\omega,s(e))}: w\in V , e\in E\}$, so is also for $q_{(\omega,A)}$ by (\ref{(9.6)}). This completes the proof.
\end{proof}


\section{Effectiveness of  $\mathcal{G}_{\mathrm{tight}}(\mathcal{S}_{G,{\mathcal{U} }})$}\label{sec4}

In this section, the effective property for the tight groupoid ${\mathcal{G}}_{\mathrm{tight}}({\mathcal{S}}_{G,\mathcal{U}})$ is investigated. Recall that a groupoid $\mathcal{G} $ is called {\it{effective}} if the interior of its isotropy group bundle $\{\alpha \in {\mathcal{G} }: s(\alpha)=r(\alpha)\}$ is equal to ${\mathcal{G} }^{(0)}$.
In order to study the effectiveness of $\mathcal{G}_{\mathrm{tight}}(\mathcal{S}_{G,{\mathcal{U} }})$, we may apply the characterization of \cite[Theorem 4.7]{kat08}, which says that, for any inverse semigroup $\mathcal{S}$ the tight groupoid $\mathcal{G}_{\mathrm{tight}}(\mathcal{S})$ is effective if and only if every interior fixed point for the maps $\theta_s :D^{s^*s}\rightarrow D^{ss^*}$, defined by $\theta_s(\mathcal{F})=(s \mathcal{F}s^*)\uparrow$, is a trivial one. So, before stating our main result, we need to compute the fixed points of $\theta_s$ for all cases of $s=(\alpha,A,g,\beta)$ in $\mathcal{S}_{G,{\mathcal{U} }}$. To this end, the computations of $\theta_s (\mathcal{F})$ for ultrafilters of the forms ${\mathcal{F}}_{(\gamma,{\mathcal{B}})}$ and ${\mathcal{F}}_x$ in \cite[Section 7]{ana02} are crucial. We recall them for convenience.

\begin{prop}[{\cite[Proposition 7.2]{ana02}}]\label{prop2.13}
Let $s=(\alpha,A,g,\omega)$ be an element of $\mathcal{S}_{G,{\mathcal{U}}}$. Then
\begin{enumerate}[$(1)$]
    \item For each $\mathcal{F}_{(\omega,\mathcal{B})}\in D^{s^*s}=\mathrm{Z} (\omega,g^{-1}\cdot A)$, we have
$$\theta_{s}(\mathcal{F}_{(\omega,\mathcal{B})})={\mathcal{F}}_ {(\alpha,g\cdot \mathcal{B}\downarrow_{s(\alpha)})},$$
where $g\cdot \mathcal{B}\downarrow_{s(\alpha)}=\{(g\cdot B)\cap s(\alpha)  :   B\in  \mathcal{B}  \}.$
    \item For each ${\mathcal{F}}_ {(\gamma,{\mathcal{B}})}\in \mathrm{Z} (\omega,g^{-1}\cdot A)$ with $\gamma\in \mathcal{U}^{\geq 1}$, we have
$$\theta_{s}({\mathcal{F}}_ {(\gamma,{\mathcal{B}})})={\mathcal{F}}_{(\alpha(g\cdot\gamma),\varphi(g,\gamma )\cdot  \mathcal{B})}.$$
\end{enumerate}
\end{prop}

\begin{prop}[{\cite[Proposition 7.3]{ana02}}]\label{prop2.14}
Let $s=(\alpha,A,g,\beta)$ be in $\mathcal{S}_{G,{\mathcal{U} }}$. For every ${\mathcal{F}}_{\beta x}\in \mathrm{Z} (\beta,g^{-1}\cdot A)$ with $x\in {\mathcal{U}}^\infty$, we have
$$\theta_{s}({\mathcal{F}}_ {\beta x})={\mathcal{F}}_ {\alpha(g\cdot x)}.$$
\end{prop}

\begin{prop}[{\cite[Proposition 7.5]{ana02}}]\label{prop7...5}
Let $s=(\alpha,A,g,\beta)$ be an element of $\mathcal{S}_{G,{\mathcal{U} }}$ with $|\beta |\geq 1 $. Then, for every tight filter ${\mathcal{F}}_ {(\beta \gamma,\mathcal{B})}\in \mathrm{Z}(\beta, g^{-1} \cdot A)$, we have
$$\theta_{s}({\mathcal{F}}_ {(\beta \gamma,\mathcal{B})})=
\begin{cases}
    {\mathcal{F}}_{\left(\omega, (g\cdot \mathcal B) \uparrow_{\mathcal {U}^0}\right)} & \text{ if} ~ \alpha=\gamma=\omega \\
    {\mathcal{F}}_ {\left(\alpha(g \cdot \gamma), \varphi(g,\gamma) \cdot \mathcal{B}\right)} & \text{ if} ~|\alpha|\geq 1 ~ \text{or} ~ |\gamma|\geq 1,
\end{cases} $$
where
$$(g\cdot \mathcal B) \uparrow_{\mathcal {U}^0}:=\{A \in \mathcal {U}^0 : \exists B \in \mathcal B \text{ such that } g\cdot B \subseteq A \}$$
is the (unique) tight filter in $\mathcal{U}^0$ containing $g \cdot \mathcal{B}$.
\end{prop}

Next, we need to determine fixed points and interior fixed points of the maps $\theta_s :D^{s^*s}\rightarrow D^{ss^*}$. Given any $s=(\alpha,A,g,\beta)$ in $\mathcal{S}_{G,{\mathcal{U} }}$, we know $D^{s^*s}=\mathrm{Z} (\beta,g^{-1}\cdot A)$ and $D^{ss^*}=\mathrm{Z} (\alpha,A)$, so in the case $\theta_s(\mathcal{F})=\mathcal{F}$, we have $\mathcal{F}\in \mathrm{Z} (\alpha,A)\cap \mathrm{Z} (\beta,g^{-1}\cdot A)$. Since such fixed points $\mathcal{F}$ are of the forms ${\mathcal{F}}_{\beta x}$ and ${\mathcal{F}}_ {(\beta \gamma,\mathcal{B})}$ by Proposition \ref{prop2.11} (see also the paragraph after it), we should consider several situations according to the cases ${\mathcal{F}}_{\beta x}$ and ${\mathcal{F}}_ {(\beta \gamma,\mathcal{B})}$ for ultrafilters.

The next definition is a generalization of \cite[Definition 3.4]{ana00} for ultragraphs.

\begin{defn}\label{defn4.1}
A $G$-{\it{cycle}} in $(G,{\mathcal{U}},\varphi)$ is a pair $(g,\gamma)\in G \times {\mathcal{U}}^{\geq 1}$ such that $g \cdot r(\gamma) \subseteq s(\gamma)$. We say a $G$-cycle $(g,\gamma)$ with $\gamma=e_1 \ldots e_n \in {\mathcal{U}}^{n}$ has an {\it{entrance}} if there exist an edge $e \in {\mathcal{U}}^{1}$ and $1\leq i \leq n$ such that $r(e)=r(e_i)$ but $e \neq e_i$.
\end{defn}

The main result of this section, Theorem \ref{thm4.8} below, will require every $G$-cycle to have an entrance. This condition is a generalization of Condition (L) for ultragraphs (c.f. \cite[Section 3]{ana00}).

Observe that, by the above definition, a $G$-cycle $(g,\gamma)$ with $\gamma=e_1 \ldots e_n \in {\mathcal{U}}^{n}$ has no entrances if and only if $s(e_i){\mathcal{U}}^{1}=\{e_{i+1}\}$ for all $1\leq i \leq n-1$, and also $s(e_n){\mathcal{U}}^{1}=\{ g \cdot {\gamma}_1\}$. In this case, we have $|s(e_i)|=1$ for $1\leq i \leq n$, because our ultragraph $\mathcal{U}$ has no sources.

Given a $G$-cycle $(g,\gamma)$, one may construct the infinite path
\begin{equation}\label{(4.2)}
x=\gamma_1 \gamma_2 \ldots
\end{equation}
such that $\gamma_1=\gamma$, $g_1=g$, and for $n \geq 2$ define  inductively $\gamma_{n+1}:=g_n\cdot \gamma_n$ and $g_{n+1}:=\varphi(g_n,\gamma_n)$. Note that, by the assumption in Definition \ref{defn2..7}(3)(b), we have
\begin{align*}
r(\gamma_{n+1})
&=g_n\cdot r(\gamma_n) \\
&\subseteq \varphi(g_{n-1},\gamma_{n-1}) \cdot s(\gamma_{n-1}) \\
&\subseteq g_{n-1}\cdot s(\gamma_{n-1}) \\
&=s(\gamma_{n}),
\end{align*}
so the infinite path (\ref{(4.2)}) is well-defined.

Let $(g,\gamma)$ be a $G$-cycle and $x$ the infinite path (\ref{(4.2)}). If $\beta \in{\mathcal{U}}^{\sharp}=\{\omega\}\cup \mathcal{U}^{\geq 1}$ with $r(\gamma)     \subseteq   s(\beta)$ and $s=(\beta\gamma,A,g,\beta)$ is an element of $\mathcal{S}_{G,{\mathcal{U} }}$, then Proposition \ref{prop2.14} implies that
\begin{align*}
{\theta}_s ({\mathcal{F}}_{\beta x})
&={\mathcal{F}}_{\beta \gamma(g \cdot x)} \\
&={\mathcal{F}}_{\beta \gamma(g \cdot \gamma_1)(\varphi(g,{\gamma}_1)\cdot \gamma_2)\cdots}  \\
&={\mathcal{F}}_{\beta( \gamma_1 \gamma_2 \cdots)}\\
&={\mathcal{F}}_{\beta x}.
\end{align*}
Therefore, ${\mathcal{F}}_{\beta x}$ is a fixed point for $\theta_s$.

\begin{lem}\label{lem4.2}
Let $s=(\alpha,A,g,\beta)$ be an element of $\mathcal{S}_{G,{\mathcal{U} }}$ such that $|\alpha|\neq|\beta|$ and let $\theta_s$ have a fixed point of the form ${\mathcal{F}}_{\beta x}$ for $x \in {\mathcal{U}}^{\infty}$.
\begin{enumerate}[$(1)$]
\item If $|\alpha|>|\beta|$, then we have $\alpha=\beta\gamma$ for some $\gamma \in  {\mathcal{U}}^{\geq 1}$, $(g,\gamma)$ is a $G$-cycle in $(G,\mathcal{U},\varphi)$ and $x$ is the unique infinite path constructed by $(g,\gamma)$ as in $(\ref{(4.2)})$.
 \item If $|\alpha|<|\beta|$, then $\beta=\alpha\gamma$ for some $\gamma \in  {\mathcal{U}}^{\geq 1}$, $(g^{-1},\gamma)$ is a $G$-cycle in $(G,\mathcal{U},\varphi)$ and $x$ is the unique infinite path constructed by $(g^{-1},\gamma)$.
\end{enumerate}
\end{lem}

\begin{proof}
To prove $(1)$, assume $|\alpha|>|\beta|$. If ${\mathcal{F}}_{\beta x}\in \mathrm{Z}(\beta,g^{-1}\cdot A) $ is a fixed point for $\theta_s$, then Proposition \ref{prop2.14} yields
\begin{equation}\label{(4..2)}
{\mathcal{F}}_{\beta x}={\theta}_s ({\mathcal{F}}_{\beta x})={\mathcal{F}}_{\alpha (g \cdot x)},
\end{equation}
so $\beta x={\alpha (g \cdot x)}$ by the unique representation of ultrafilters in Proposition \ref{prop2.11}. As $|\alpha|>|\beta|$, $\beta$ must be an initial subpath of $\alpha$, i.e. $\alpha=\beta\gamma$ for some $\gamma\in {\mathcal{U}}^{\geq 1}$. In particular, $\gamma$ is an initial segment of $x$ such that
$$g\cdot  r(\gamma)=g \cdot r(x)\subseteq A \subseteq s(\alpha)=s(\gamma).$$
This says that $(g,\gamma)$ is a $G$-cycle. Furthermore, the fact $\beta x={\alpha (g \cdot x)}=\beta  \gamma (g\cdot x)$ deduces $x=\gamma (g\cdot x)$. Write $x=\gamma_1 \gamma_2 \ldots$ with $|\gamma_i|=|\gamma|$ for all $i\geq 1$. Then $\gamma_1=\gamma$ and the decomposition
$$\gamma_1 \gamma_2 \ldots=x= \gamma (g\cdot x)= \gamma (g\cdot \gamma_1 )(\varphi(g,\gamma_1)\cdot \gamma_2 )\ldots$$
concludes that $x$ is precisely the infinite path constructed in $(\ref{(4.2)})$.

The proof of statement (2) is analogous to that of (1).
\end{proof}

\begin{lem}\label{lem4.3}
If $s=(\alpha,A,g,\beta)$ is an element of $\mathcal{S}_{G,{\mathcal{U} }}$ with $|\alpha|\neq|\beta|$, then $\theta_s$ has no fixed points of the form ${\mathcal{F}}_{(\gamma,\mathcal{B} )}$.
\end{lem}

\begin{proof}
If ${\mathcal{F}}_{(\gamma,\mathcal{B} )}\in D^{s^*s}=\mathrm{Z}(\beta,g^{-1}\cdot A)$, we necessarily have $\gamma=\beta\delta$ for some $\delta \in {\mathcal{U} }^{\sharp}$. Note that Proposition \ref{prop7...5} says
$$\theta_{s}({\mathcal{F}}_ {(\gamma,{\mathcal{B}})})=
\begin{cases} {\mathcal{F}}_{\left(\omega, (g\cdot \mathcal B) \uparrow_{\mathcal {U}^0}\right)}\quad\quad\quad\qquad\;\;\text{ if}\;\;\alpha=\delta=\omega \\{\mathcal{F}}_ {(\alpha(g\cdot\delta),\varphi(g,\delta )\cdot {\mathcal{B}})}\qquad\qquad\qquad \text{otherwise}.
\end{cases}$$
So, in the case $\alpha=\delta=\omega$, if $\mathcal{F}_{(\gamma,\mathcal{B})}$ is a fixed point then $\gamma=\omega$ by uniqueness of the representation and therefore we also have $\beta=\omega$ (because $\gamma=\beta\delta$). Hence, this case could not occur by the assumption $|\alpha|\neq |\beta|$. Moreover, in the other case, since $|\alpha|\neq|\beta|$, we have $\beta\delta\neq \alpha(g\cdot \delta)$, whence the map $\theta_{s}$ does not fix ${\mathcal{F}}_ {(\gamma,{\mathcal{B}})}$.
\end{proof}

We conclude the following proposition for the fixed points of $\theta_s$ whenever $|\alpha|\neq|\beta|$.

\begin{prop}\label{prop4.4}
Suppose that every $G$-cycle in $(G,\mathcal{U},\varphi)$ has an entrance. Then for each $s=(\alpha,A,g,\beta)$ in $\mathcal{S}_{G,{\mathcal{U} }}$ with $|\alpha|\neq|\beta|$, $\theta_s$ has no interior fixed points.
\end{prop}

\begin{proof}
Fixing such $s=(\alpha,A,g,\beta)$, we first assume $|\alpha|>|\beta|$. By Lemma \ref{lem4.3}, $\theta_s$ has no fixed points of finite type ${\mathcal{F}}_ {(\gamma,{\mathcal{B}})}$. Assume, by way of contradiction, that ${\mathcal{F}}_ {{\beta x}}$ is an interior fixed point for $\theta_s$. Then Lemma \ref{lem4.2} says that $\alpha=\beta\gamma$ for some $\gamma \in  {\mathcal{U}}^{\geq 1}$ so that $(g,\gamma)$ is a $G$-cycle and $x$ is uniquely constructed by $(g,\gamma)$ as in $(\ref{(4.2)})$. Hence, ${\mathcal{F}}_ {{\beta x}}$ is isolated and there exists a cylinder $\mathrm{Z}(\beta\lambda ,B)$ satisfying $\mathrm{Z}(\beta\lambda ,B)=\{ {\mathcal{F}}_ {{\beta x}}\}$. (It follows from Proposition \ref{prop2.11} that basically open sets containing $\mathcal{F}_{\beta x}$ are of the form of $Z(\beta \lambda,B)$ such that $\lambda$ is an initial subpath of $x$ and $B\subseteq s(\lambda)$.) But it follows that the infinite path $\beta x$ has no entrance, so does the $G$-cycle $(g,\gamma)$ because $x$ is constructed by $(g,\gamma)$. This is a contradiction; therefore, such a fixed point ${\mathcal{F}}_ {{\beta x}}$ does not exist. The proof for the case  $|\alpha|<|\beta|$ is analogous.
\end{proof}

Next, we want to describe the fixed points for elements $s=(\alpha,A,g,\beta)$ with $|\alpha|=|\beta|$.

\begin{lem}\label{lem4.5}
Let $s=(\alpha,A,g,\beta)$ be an element of $\mathcal{S}_{G,{\mathcal{U} }}$ with $|\alpha|=|\beta|$. If $\theta_s$ has a fixed point, then $\alpha=\beta$ and all fixed points of $\theta_s$ are of the following forms:
\begin{enumerate}[$(1)$]
\item ${\mathcal{F}}={\mathcal{F}}_{\alpha x}$ for $x \in (g^{-1}\cdot A){\mathcal{U} }^{\infty}$ satisfying $g \cdot x=x$.
 \item ${\mathcal{F}}={\mathcal{F}}_{(\alpha\gamma,\mathcal{B})}$ for $\gamma \in (g^{-1}\cdot A){\mathcal{U} }^{\sharp}$ satisfying $g \cdot \gamma=\gamma$ and $\varphi(g,\gamma)\cdot \mathcal{B}=\mathcal{B}$.
\end{enumerate}
\end{lem}

\begin{proof}
We know that ultrafilters in $D^{s^*s}=\mathrm{Z}(\beta,g^{-1}\cdot A)$ are of the forms ${\mathcal{F}}_{\beta x}$ and ${\mathcal{F}}_{(\beta\gamma,\mathcal{B})}$. In the case $\theta_s({\mathcal{F}}_{\beta x})={\mathcal{F}}_{\beta x}$, Proposition \ref{prop2.14} implies that ${\mathcal{F}}_{\beta x}={\mathcal{F}}_{\alpha(g \cdot x)}$, so ${\beta x}={\alpha(g \cdot x)}$. As assumed $|\alpha|=|\beta|$, we obtain $\alpha=\beta$ and $g\cdot x=x$, concluding statement (1).

In the other case, if $\theta_s({\mathcal{F}}_{(\beta\gamma,\mathcal{B})})={\mathcal{F}}_{(\beta\gamma,\mathcal{B})}$, then Propositions \ref{prop2.13} and \ref{prop7...5} yield ${\mathcal{F}}_{(\beta\gamma,\mathcal{B})}={\mathcal{F}}_{(\alpha(g\cdot \gamma),\varphi(g,\gamma)\cdot\mathcal{B})}$, and hence $\beta\gamma=\alpha(g\cdot \gamma)$ and $\mathcal{B}=\varphi(g,\gamma) \cdot \mathcal{B}$. Again, since $| \alpha |=|  \beta  |$, it follows that $\alpha=\beta$ and $\gamma=g \cdot \gamma$ proving statement (2).
\end{proof}

Finally, we provide conditions for a self-similar ultragraph $(G,\mathcal{U},\varphi)$ to ensure that every interior fixed point of maps $\theta_s:D^{s^* s}    \rightarrow   D^{ss^* }$ is a trivial one. They will generalize conditions $(1),(3),(4)$ of \cite[Theorem 3.11]{apl00} together with those of \cite[Theorem 14.10]{bro14}.

In the following, we use the notations $A{\mathcal{U}^\infty}:=\{x\in \mathcal{U}^\infty : r(x)\subseteq A       \}$ and
$$A{\mathcal{U}^{\leq \infty}}:=A \mathcal{U}^\infty \cup \{\alpha\in \mathcal{U}^{\geq 1} : r(\alpha)\subseteq A   \text{ and } |s(\alpha)|=\infty\}.$$
Moreover, for $g\in G$, we say a path $\alpha$ with $|\alpha|\geq 1$ is {\it strongly fixed by $g$} if $g\cdot \alpha=\alpha$ and $\varphi(g,\alpha)=1_G$.

\begin{defn}[{\bf Condition $(*)$}]\label{defn4.9}
We say that $(G,\mathcal{U},\varphi)$ satisfies {\it{Condition $(*)$}} if for $g\in G\setminus \{1_G\}$ and $A\in{\mathcal{U} }^0$ satisfying $g\cdot x=x$ for all $x\in A {\mathcal{U} }^\infty$, we then have
\begin{enumerate}[$(1)$]
\item every $x\in {A{\mathcal{U} }}^{\leq \infty}$ has an initial subpath $\alpha$ (i.e. $x=\alpha y$) such that it is strongly fixed by $g$;
\item in the case $|A|=\infty$, there are no ultrafilters in ${\mathcal{U} }^0$  (with $\cap$ as the meet and $\subseteq$ as the ordering for $\mathcal{U}^0$) containing $A$.
\end{enumerate}
\end{defn}

We also need a simple lemma:

\begin{lem}\label{lem4.7}
Let $\mathcal{B}$ be a filter in ${\mathcal{U} }^0$ and $g \in G$. If $\mathcal{B}$ contains a set $A$ such that $g\cdot v=v$ for all $v\in A$, then we have $g\cdot \mathcal{B}=\mathcal{B}$.
\end{lem}

\begin{proof}
Picking an arbitrary set $B \in \mathcal{B}$, we will show that $g\cdot B \in \mathcal{B}$. Since $\mathcal{B}$ is a filter, we have $\emptyset \neq A \cap B \in \mathcal{B}$. As we assumed $g\cdot (A \cap B )=A \cap B $, then
$$A \cap B=g\cdot (A \cap B ) \subseteq g\cdot B .$$
So, the fact $A\cap B\in \mathcal{B}$ concludes $g\cdot B\in \mathcal{B}$ because $\mathcal{B}$ is a filter.
\end{proof}

We are now in the position to state and prove the main result of this section.

\begin{thm}\label{thm4.8}
Let $(G,\mathcal{U},\varphi)$ be a self-similar ultragraph. If every $G$-cycle in $(G,\mathcal{U},\varphi)$ has an entrance and $(G,\mathcal{U},\varphi)$ satisfies Condition $(*)$, then $\mathcal{G}_{\mathrm{tight}}(\mathcal{S}_{G,{\mathcal{U} }})$ is effective.
\end{thm}

\begin{proof}
In light of \cite[Theorem 4.7]{kat08}, effectiveness of $\mathcal{G}_{\mathrm{tight}}(\mathcal{S}_{G,{\mathcal{U} }})$ is equivalent to the action $\theta:\mathcal{S}_{G,{\mathcal{U} }}   \curvearrowright \mathcal{T}$  being topologically free. That is, for any  $s \in \mathcal{S}_{G,{\mathcal{U} }}$, every interior fixed point of $\theta_s$ is a trivial fixed point.\footnote{A fixed point $\mathcal{F}$ for $\theta_s$ is called {\it trivial} whenever there exists an idempotent $q\leq s$ such that $\theta_q(\mathcal{F})= \mathcal{F}$.} So, pick an arbitrary element $s=(\alpha,A,g,\beta)$ in $\mathcal{S}_{G,{\mathcal{U} }}$ and assume that $\mathcal{F}\in D^{s^* s}=\mathrm{Z}(\beta,g^{-1}\cdot A)$ is an  interior fixed point for $\theta_s$. Proposition \ref{prop4.4} and Lemma \ref{lem4.5} together turn out $\alpha=\beta$, i.e. $s=(\alpha,A,g,\alpha)$. Thus there exists a cylinder $\mathrm{Z}(\alpha\lambda,B)$ containing $\mathcal{F}$ such that $\theta_s$ fixes all tight filters in $\mathrm{Z}(\alpha\lambda,B)$. Note that Lemma \ref{lem4.5} above describes all fixed points of $\theta_s$, which are of the forms $\mathcal{F}_{(\alpha\lambda\gamma,\mathcal{B}  )}$ and $\mathcal{F}_{\alpha\lambda  x  }$. Hence, it suffices to show that such points are trivial ones.

Note that for the tight filters of the form $\mathcal{F}_{\alpha\lambda  x  }\in \mathrm{Z}(\alpha\lambda,B)$, Lemma \ref{lem4.5}(1) yields $g\cdot (\lambda x)=\lambda x$, or equivalently $g\cdot \lambda=\lambda$ and $\varphi(g,\lambda)\cdot x=x$. If we write $h:=\varphi(g,\lambda)$, it says that $h\cdot x=x$ for all $x\in {\mathcal{U}^\infty}$ with $r(x)\subseteq B$, and we may apply properties (1) and (2) in Condition $(*)$ for such $h \in G$ and $B\in {\mathcal{U} }^0$.

First, property (2) in Condition $(*)$ deduces that there are no tight filters of the form $\mathcal{F}_{(\alpha\lambda,\mathcal{B}  )}$ in $\mathrm{Z}(\alpha\lambda,B)$. To prove this claim, assume on the contrary that there exists some $\mathcal{F}_{(\alpha\lambda,\mathcal{B}  )}$ in $\mathrm{Z}(\alpha\lambda,B)$, where $\mathcal{B}$ is an ultrafilter in $\mathcal{B} _{s(\lambda)}$ containing $B$. In this case, we must have $|B|=\infty$ by Proposition \ref{prop2.11}. According to \cite[Proposition 7.5]{ana02}, we may enlarge $\mathcal{B}$ to the ultrafilter
$$\mathcal{B}\uparrow_{{\mathcal{U}^0}}:=\{ C\in \mathcal{U}^0 : \exists C'\in \mathcal{B}\; \text{such that }  C'\subseteq C\}$$
in $\mathcal{U}^0$ containing $\mathcal{B}$. On the other hand, we have $h\cdot v=v$ for all $v\in B$ because every $v\in B$ receives an infinite path $x\in B \mathcal{U}^{\infty}$, and $h\cdot x=x$ turns out
$$ h\cdot v=h \cdot r(x)=r(h\cdot x)=r(x)=v. $$
Thus, Lemma \ref{lem4.7} implies $h\cdot \mathcal{B}\uparrow_{{\mathcal{U}^0}}=\mathcal{B}\uparrow_{{\mathcal{U}^0}}$, contradicting the property (2) in Condition $(*)$.

Now, we will show that all tight filters of the forms $\mathcal{F}_{(\alpha\lambda\gamma,\mathcal{B}  )}$ with $\gamma \in {\mathcal{U}}^{\geq 1}$ and $\mathcal{F}_{\alpha\lambda  x  }$ in $\mathrm{Z}(\alpha\lambda,B)$ are trivial. In case  $\mathcal{F}_{(\alpha\lambda\gamma,\mathcal{B}  )}\in \mathrm{Z}(\alpha\lambda,B)$, then $\gamma \in B{\mathcal{U} }^{ \geq 1}$ satisfies $g\cdot (\lambda\gamma)=\lambda\gamma$ by Lemma \ref{lem4.5}(2). It follows $(g\cdot \lambda)(\varphi(g,\lambda)\cdot \gamma) =\lambda \gamma$, and since $|  g\cdot \lambda |=|  \lambda  |$, we get $\varphi(g,\lambda)\cdot \gamma=\gamma$ (i.e., $h \cdot \gamma=\gamma$). By property $(1)$ in Condition $(*)$, $\gamma$ would have an initial subpath that is strongly fixed by $h$, so is $\gamma$ itself. It is straightforward to check that $q:=q_{( \alpha\lambda\gamma ,s(\gamma)  )}\leq s$ and $\theta_q$ fixes $\mathcal{F}_{(\alpha\lambda\gamma,\mathcal{B})}$ by Propositions \ref{prop2.13} and \ref{prop7...5}. Therefore, $\mathcal{F}_{(\alpha\lambda\gamma,\mathcal{B}  )}$ is a trivial fixed point for $\theta_s$.

The argument for the tight filters $\mathcal{F}_{\alpha\lambda  x}$ in $\mathrm{Z}(\alpha\lambda,B)$ is analogous. Indeed, by the property (1) in Condition $(*)$, $x$ has an initial subpath, say $\eta$, which is strongly fixed by $h$. Then $q:=q_{(  \alpha\lambda\eta ,s(\eta) )}\leq s$ and Proposition \ref{prop2.14} implies that $\theta_q(\mathcal{F}_{\alpha\lambda  x})=\mathcal{F}_{\alpha\lambda  x}$. Consequently, we have already shown that all tight filters in $\mathrm{Z}(\alpha\lambda,B)$ are trivial ones, completing the proof.
\end{proof}


\section{Examples}\label{sec5}
In this section, we verify conditions in Theorem \ref{thm3.4} and \ref{thm4.8} for three examples of self-similar ultragraphs.
\begin{example}
Let $(\mathbb{Z},{\mathcal{U}},\varphi)$ be the self-similar ultragraph in \cite[Example 3.4]{ana02}. It is the ultragraph

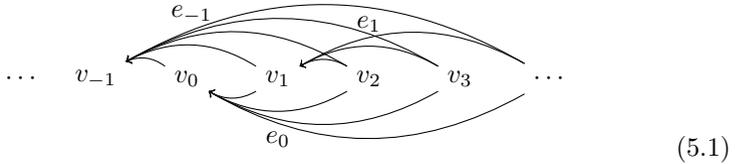
\begin{equation}\label{ex5.1}
\begin{tikzpicture}[scale=0.8]
\node (...) at (-1.2,0) {$\cdots$};
\node (v_{-1}) at (0,0) {$v_{-1}$};
\node (v_0) at (1.5,0) {$v_0$};
\node (v_1) at (3,0) {$v_1$};
\node (v_2) at (4.5,0) {$v_2$};
\node (v_3) at (6,0) {$v_3$};

\node (00) at (7.5,0) {$\cdots$};

\path[->] (00) edge[bend right]   (v_{-1});
\path[->] (v_0) edge[bend right]  (v_{-1});
\path[->] (v_1) edge[bend right] node[above=5pt]  {$e_{-1}$}(v_{-1});
\path[->] (v_2) edge[bend right]  (v_{-1});
\path[->] (v_2) edge[bend right]  (v_1);
\path[->] (00)  edge[bend left]   (v_0);
\path[->] (v_1) edge[bend left]   (v_0);
\path[->] (v_2) edge[bend left] node[above=-17pt]  {$e_0$} (v_0);
\path[->] (v_3) edge[bend right] node[above=2pt]  {$e_1$} (v_1);
\path[->] (v_3) edge[bend right]   (v_{-1});
\path[->] (00) edge[bend right]   (v_1);
\path[->] (v_3) edge[bend left]   (v_0);
\end{tikzpicture}
\end{equation}
where $ U^0=\{v_i : i\in \mathbb{Z}\}$ and ${\mathcal{U}}^1=\{e_i :  i\in \mathbb{Z}\}$ with the range map $r(e_i)=\{v_i\}$ and the source map $s(e_i)=\{v_j : j>  i\}$ for $i\in \mathbb{Z}$. The action $\mathbb{Z}\curvearrowright {\mathcal{U}}$ is defined by
$$n \cdot v_i=v_{i+n} \;\;\text{     and     }\; \;n \cdot e_i=e_{i+n}\;\; (\forall \; i,n \in \mathbb{Z}),$$
and let $\varphi:  \mathbb{Z}\times {\mathcal{U}}^1 \rightarrow \mathbb{Z}$ be a $1$-cocycle satisfying
$$\varphi(n,e_i)\cdot s(e_i) \subseteq n \cdot s(e_i)\;\; (\forall \; i,n \in \mathbb{Z})$$
(you may consider the trivial $1$-cocycle $\varphi(n,e_i)=n$ for instance).
We see in figure (\ref{ex5.1}) that every pair $(n,e_i) \in \mathbb{N}\times {\mathcal{U}}^1$ with $n>  i$, is a $G$-cycle in $(\mathbb{Z},{\mathcal{U}},\varphi  )$ having entrances. Moreover, for each $n\neq 0$, we have $n \cdot v_i  \neq v_i$, and hence $n \cdot x \neq x$ for all $x\in {\mathcal{U}}^\infty$. Therefore, the self-similar ultragraph $(\mathbb{Z},{\mathcal{U}},\varphi  )$ satisfies Condition $(*)$, and together with the fact that every $G$-cycle has an entrance, Theorem \ref{thm4.8} implies that the groupoid $\mathcal{G}_{\mathrm{tight}}({S}_{\mathbb{Z},{\mathcal{U} }})$ is effective (or equivalently, the action $\theta:\mathcal{S}_{\mathbb{Z},{\mathcal{U} }}\curvearrowright \widehat{\mathcal{E}}_{\mathrm{tight}}(\mathcal{S}_{\mathbb{Z},\mathcal{U}})$ is topologically free).

Furthermore, given every $v_n\in U^0$ and $x=\alpha_1 \alpha_2 \ldots \in {\mathcal{U}}^\infty$, if $r(x)=v_m$ then $(m-n)\cdot v_n=v_m=r(x)$. It follows that $(\mathbb{Z},{\mathcal{U}},\varphi  )$ is $\mathbb{Z}$-cofinal in the sense of Definition \ref{defn3..1}, and therefore $\mathcal{G}_{\mathrm{tight}}(\mathcal{S}_{\mathbb{Z},\mathcal{U}})$ is a minimal groupoid by Theorem \ref{thm3.4}.
\end{example}

\begin{example}\label{ex5.2}
Let $(\mathbb{Z}_{2},{\mathcal{U}},\varphi  )$ be a self-similar ultragraph in which $\mathbb{Z}_{2}=(\{0,1\},+)$ and ${\mathcal{U}}$ is the ultragraph

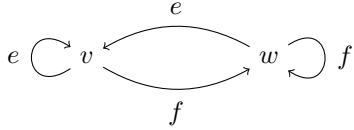
\begin{center}
\begin{tikzpicture}[scale=0.6]

\node (v) at (-1.5,0) {$v$};
\node (w) at (2.5,0) {$w$};

\path[<-] (v) edge[bend left] node [above=1pt]{$e$} (w);
\path[<-] (w) edge[bend left] node [above=-17pt]{$f$} (v);

\draw[<-] (w) ..controls (4,-1) and (4,1) .. node[right=1pt] {$f$} (w);

\draw[->] (v) ..controls (-3,-1) and (-3,1) .. node[left=1pt] {$e$} (v);
\end{tikzpicture}
\end{center}
with the action $\mathbb{Z}_{2}\curvearrowright \mathcal{U}$ by $0 \cdot \alpha=\alpha$ for all $\alpha \in U^0 \cup {\mathcal{U}}^1$ and
$$1\cdot v=w , \;\;\; \;\;\;\;1\cdot w=v,$$
$$1\cdot e=f \;\;\text{and} \;\; 1\cdot f=e.$$
Since there are edges from $v$ to $w$ and from $w$ to $v$, $(\mathbb{Z}_{2},{\mathcal{U}},\varphi  )$ is $G$-cofinal. Moreover, we have $1\cdot v\neq v$  and $1 \cdot w\neq w$, so $1\cdot x\neq x$ for all $x\in {\mathcal{U}}^{\infty}$. Hence, $(\mathbb{Z}_{2},{\mathcal{U}},\varphi  )$ trivially satisfies Condition $(*)$. Since the $G$-cycles $(0,e),(0,f),(1,e),(1,f)$ have entrances, Theorems \ref{thm3.4} and \ref{thm4.8} conclude that the groupoid $\mathcal{G}_{\mathrm{tight}}({S}_{\mathbb{Z}_{2},{\mathcal{U} }})$ is minimal and effective.
\end{example}

From now on, and until the end of this section, we consider the following self-similar ultragraph with an arbitrary, but fixed, 1-cocycle $\varphi$.

\begin{example}\label{ex5.3}
Let $\mathcal{U}$ be the ultragraph
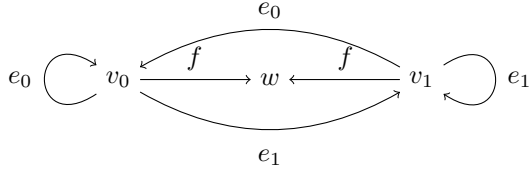
\begin{center}
\begin{tikzpicture}[scale=0.8]

\node (v0) at (-1.5,0) {$v_0$};
\node (v1) at (3.5,0) {$v_1$};
\node (w) at (1,0) {$w$};

\path[<-] (v0) edge[bend left] node [above=1pt]{$e_0$} (v1);
\path[<-] (v1) edge[bend left] node [above=-17pt]{$e_1$} (v0);

\draw[->] (v0) ..controls (-.25,0) and (-.25,0) .. node[above=0pt]{$f$} (w);
\draw[->] (v1) ..controls (2.25,0) and (2.25,0) .. node[above=0pt]{$f$} (w);

\draw[<-] (v1) ..controls (5,-1) and (5,1) .. node[right=1pt] {$e_1$} (v1);
\draw[->] (v0) ..controls (-3,-1) and (-3,1) .. node[left=1pt] {$e_0$} (v0);
\end{tikzpicture}
\end{center}
If $G$ is the additive group $(\mathbb{Z},+)$, define the action $\mathbb{Z}\curvearrowright \mathcal{U}$ by $n\cdot a=a$ for $a\in \{w,f\}$ and
$$\begin{array}{cc}
  n\cdot v_0=v_{[n]_2}, & n\cdot v_1=v_{[n+1]_2}, \\
  n\cdot e_0=e_{[n]_2}, & n\cdot e_1=v_{[n+1]_2},
\end{array}
$$
where $[n]_2\in \{0,1\}$ is denoted for $n$ modulo $2$. We fix a 1-cocycle $\varphi$ such that $(\mathbb{Z},\mathcal{U},\varphi)$ is a self-similar ultragraph.
\end{example}

As seen in the figure, $\mathcal{U}$ contains two loops $e_0$, $e_1$ having entrances (each of $e_0$ and $e_1$ is an entrance for the other), and every $G$-cycle in $\mathcal{U}$ has an entrance. Moreover, $(\mathbb{Z},\mathcal{U},\varphi)$ is $G$-cofinal, so it remains to verify Condition $(*)$ which depends on the values of $\varphi$. In order to do this, we assume a number $n\in \mathbb{Z}\setminus \{0\}$ and a set $A=\{v_0\}$ are given such that $n\cdot x=x$ for all $x\in \{v_0\}\mathcal{U}^\infty$. (Note that the argument for $A=\{v_1\}$ is analogous with same computations.) Then, one may compute that:

\begin{lem}\label{lem5.4}
Consider the self-similar ultragraph $(\mathbb{Z},\mathcal{U},\varphi)$ of Example \ref{ex5.3}. If $n\in \mathbb{Z}\setminus \{0\}$ is given such that $n \cdot x=x$ for all $x\in  \mathcal{U}^\infty$ with $r(x)=\{v_0\}$, then
\begin{enumerate}[$(1)$]
  \item $n$ and $\varphi(n,\alpha)$ are even integers for every $\alpha\in \mathcal{U}^{\geq 1}$ with $r(\alpha)=\{v_0\}$.
  \item $\varphi(1,e_0)+\varphi(1,e_1)$ is an even integer.
  \item For every $\alpha\in \mathcal{U}^{\geq 1}$ with $r(\alpha)=\{v_0\},\{v_1\}$, we have
  $$\varphi(n,\alpha)=n t^{|\alpha|},$$
  where $t:=\frac{1}{2}(\varphi(1,e_0)+\varphi(1,e_1))$.
\end{enumerate}
\end{lem}

\begin{proof}
(1). Observe that $n\cdot x=x$ implies
$$x=(-n+n)\cdot x=(-n)\cdot (n\cdot x)=(-n)\cdot x,$$
so, without loss of generality, we may assume $n>0$. It follows also $(n\cdot \alpha)(\varphi(n,\alpha)\cdot y)=\alpha y$, whence $\varphi(n,\alpha)\cdot y=y$ for every decomposition $x=\alpha y$ of $x$. In particular, we have $n\cdot r(x)=r(x)$, which forces $n\in 2 \mathbb{Z}$ by the definition of $\mathbb{Z} \curvearrowright U^0$. Similarly, $\varphi(n,\alpha)\cdot y=y$ implies $\varphi(n,\alpha) \cdot r(y)=r(y)$, and thus $\varphi(n,\alpha)\in 2\mathbb{Z}$ for every $\alpha\in \{v_0\}\mathcal{U}^{\geq 1}$.

(2) and (3). Let us denote $t_0:=\varphi(1,e_0)$ and $t_1:=\varphi(1,e_1)$ for convenience. Then, apply the 1-cocycle property (Definition \ref{defn2..7}(3)(a)) to get
\begin{align*}
\varphi(2,e_0)&=\varphi(1+1,e_0)\\
&=\varphi(1,1\cdot e_0)+ \varphi(1,e_0)\\
&=\varphi(1,e_1)+\varphi(1,e_0)\\
&=t_0+t_1,
\end{align*}
and similarly, $\varphi(2,e_1)=t_0+t_1$. Moreover, for each $i=0,1$,
\begin{align*}
\varphi(2+2,e_i)&=\varphi(2,2\cdot e_i)+ \varphi(2,e_i)\\
&=\varphi(2, e_i)+ \varphi(2,e_i)\\
&=2(t_0+t_1).
\end{align*}
Continuing this process gives
\begin{equation}\label{eq5.2}
\varphi(2k,e_i)=k\varphi(2,e_i)=k(t_0+t_1)
\end{equation}
for every $k\geq 1$ and $i=0,1$.

Now, we want to compute $\varphi(n,\alpha)$ for paths $\alpha=\alpha_1 \alpha_2 \ldots \alpha_k$, with $\alpha_i\in \{e_0,e_1\}$.\footnote{Here, we should assume $\alpha_1=e_0$, but we can extend the case to any $\alpha_1\in \{e_0,e_1\}$ because $\varphi(n,e_0)=\varphi(n,e_1)$.}
Write $n=2k_0$ by statement (1). Then, the equality (\ref{eq5.2}) for $k=k_0$ says $\varphi(n,e_i)=k_0(t_0+t_1)$,
which is an even number by statement (1). If we write $k_0(t_0+t_1)=2k_1$, then for each $\alpha=\alpha_1\alpha_2$ with $\alpha_1,\alpha_2\in \{e_0,e_1\}$, we have
$$\varphi(n,\alpha)=\varphi(\varphi(n,\alpha_1), \alpha_2)=\varphi(2k_1,\alpha_2)\overset{(\ref{eq5.2})}{=} k_1(t_0+t_1)=\frac{k_0}{2}(t_0+t_1)^2,$$
which is again even by (1). Similarly, if $\frac{k_0}{2}(t_0+t_1)^2=2k_2$, then
$$\varphi(n,\alpha_1\alpha_2\alpha_3)=k_2(t_0+t_1)= \frac{k_0}{2^2}(t_0+t_1)^3,$$
and we may inductively obtain
\begin{equation}\label{eq5.3}
\varphi(n,\alpha_1 \ldots \alpha_k)=\frac{k_0}{2^{k-1}}(t_0+t_1)^k.
\end{equation}
As the number in (\ref{eq5.3}) is even for every $k\geq 1$, then so is $t_0+t_1$ as well, and consequently statements (2) and (3) are proved.
\end{proof}

Note that, by similarity, if we have $0\neq n\in \mathbb{Z}$ such that $n \cdot x=x$ for all $x\in \mathcal{U}^\infty$ with $r(x)=\{v_1\}$, then statements (1)-(3) in the above lemma also hold by replacing $v_0$ with $v_1$.

 Using the above lemma, we can precisely determine when the self-similar ultragraph of Example \ref{ex5.3} satisfies Condition $(*)$.

\begin{prop}\label{prop5.5}
The self-similar ultragraph $(\mathbb{Z},\mathcal{U},\varphi)$ of Example \ref{ex5.3} satisfies Condition $(*)$ if and only if
$$\varphi(1,e_0)+\varphi(1,e_1)\in \{0\} \cup (2\mathbb{Z}+1).$$
\end{prop}

\begin{proof}
Observe that each infinite path $x\in \mathcal{U}^\infty$ with $r(x)=\{w\}$ is of the form $x=fy$ for $y\in \{v_0,v_1\}\mathcal{U}^\infty$. So, for any $n\in \mathbb{Z}\setminus \{0\}$, $n\cdot x=x$ implies $(n\cdot f)(\varphi(n,f) \cdot y)=fy$ or equivalently, $\varphi(n,f)\cdot y=y$ because $n \cdot f=f$. In the case $\varphi(n,f)\in 2\mathbb{Z}+1$, we have $\varphi(n,f) \cdot r(y) \neq r(y)$ by definition, which contradicts $\varphi(n,f) \cdot y=y$. Moreover, if $\varphi(n,f)=0$, then $x$ has the prefix $f$ strongly fixed by $n$, so property (1) of Definition \ref{defn4.9} is satisfied.\footnote{Notice that property (2) in Definition \ref{defn4.9} is trivially satisfied because $|A|<\infty$ for all $A\in \mathcal{U}^0$.} Hence, checking Condition ($*$) in Definition \ref{defn4.9} for such $x=fy\in \{w\}\mathcal{U}^\infty$ and $n\in \mathbb{Z}\setminus \{0\}$ reduces to the case $\varphi(n,f)\in 2\mathbb{Z}\setminus \{0\}$ and $\varphi(n,f)\cdot y=y$ for all $y\in \{v_0,v_1\}\mathcal{U}^\infty$.

Therefore, in order to verify Condition $(*)$, it suffices to do it for $n\in 2 \mathbb{Z}\setminus \{0\}$ and the sets $A=\{v_0\},\{v_1\}$. Let $A=\{v_0\}$ and let $n \in 2\mathbb{Z}\setminus \{0\}$ be given such that $n \cdot x=x$ for all $x\in \mathcal{U}^\infty$ with $r(x)=\{v_0\}$. Then, by Lemma \ref{lem5.4}(3), every such $x$ has an initial subpath strongly fixed by $n$ if and only if $\varphi(1,e_0)+\varphi(1,e_1)=0$ (in the case, $e_0, e_1$ are strongly fixed by $n$ as well, cf. eq. (\ref{eq5.2})).

Furthermore, if $\varphi(1,e_0)+\varphi(1,e_1)$ is an odd integer, then Lemma \ref{lem5.4}(2) implies that the condition ``$n\cdot x=x$ for all $x\in \{v_0\}\mathcal{U}^\infty$" does not occur for all $n\in \mathbb{Z}\setminus \{0\}$, and $(\mathbb{Z}, \mathcal{U},\varphi)$ trivially satisfies Condition $(*)$ in this case.

A same argument could be said for $A=\{v_1\}$, and hence the result is deduced.
\end{proof}


\section{A special case: crossed product and simplicity}\label{sec6}

In this section, we consider self-similar ultragraphs $(G,{\mathcal{U}},\varphi  )$ in which the $1$-cocycle $\varphi$ is the trivial one, that is $\varphi(g,\alpha):=g$ for all $g\in G$ and $\alpha \in {\mathcal{U}}^*$. We will denote such 1-cocycle by $\iota$. In this case, we show that the action $G\curvearrowright \mathcal{U}$ induces a $C^*$-dynamical system $\eta:G \curvearrowright C^*({\mathcal{U}})$, and then the $C^*$-algebra $\mathcal{O}_{G,{\mathcal{U} }}$ is isomorphic to the crossed product $C^*({\mathcal{U}})\rtimes_{\eta} G$. This deduces, in particular, a result for the simplicity of $\mathcal{O}_{G,{\mathcal{U} }}$.

\subsection{A unitary representation of $G$}\label{subs6.1}

Let $(G,\mathcal{U},\varphi)$ be a self-similar ultragraph as in Definition \ref{defn2..7}. In this short subsection, we briefly provide another, but equivalent, definition for the $C^*$-algebra $\mathcal{O}_{G,\mathcal{U}}$ which is more analogous to that of $\mathcal{O}_{G,E}$ in the self-similar graph setting \cite{bro14,exe18}. Using this definition, we may define a unitary $*$-representation of $G$ on the multiplier algebra $M(\mathcal{O}_{G,\mathcal{U}})$.

\begin{defn}\label{defn6.1}
Let $\mathcal{D}_{G, \mathcal{U}}$ be the {\it universal unital $C^*$-algebra} generated by a family
$$\{s_e,p_A:A\in \mathcal{U}^0,e\in \mathcal{U}^1\} \cup \{u_g:g\in G\}$$
satisfying the following properties:
\begin{enumerate}[(1)]
  \item $\{s_e,p_A:A\in \mathcal{U}^0,e\in \mathcal{U}^1\}$ is a Cuntz-Krieger $\mathcal{U}$-family;
  \item $u: G \rightarrow \mathcal{D}_{G, \mathcal{U}}$, defined by $g\mapsto u_g$, is a unitary $*$-representation of $G$;
  \item \begin{enumerate}[(a)]
          \item $u_g p_A= p_{g \cdot A} u_g$, and
          \item $u_g s_e= s_{g \cdot e} u_{\varphi(g,e)}$ for all $g\in G$, $A\in \mathcal{U}^0$, and $e\in \mathcal{U}^1$.
        \end{enumerate}
\end{enumerate}
Then, we define the $C^*$-subalgebra
\begin{equation}\label{eq6.1-1}
\widetilde{\mathcal{O}}_{G,\mathcal{U}}:= \overline{\mathrm{span}}\{s_\alpha p_A u_g s_\beta^*: \alpha,\beta\in \mathcal{U}^*, A\in \mathcal{U}^0, A\subseteq s(\alpha)\cap g \cdot s(\beta)\}
\end{equation}
of $\mathcal{D}_{G, \mathcal{U}}$, where $s_\alpha:=p_A$ if $\alpha=A\in \mathcal{U}^0$ and $s_\alpha:=s_{e_1}\ldots s_{e_n}$ if $\alpha=e_1\ldots e_n\in \mathcal{U}^n$ for $n\geq 1$.
\end{defn}

The following proposition links Definition \ref{defn6.1} above to Definition \ref{defn2...3} for $\mathcal{O}_{G,\mathcal{U}}$.

\begin{prop}\label{prop6.2}
Let
$$\{S_e,P_A,V_{A,g}:A\in\mathcal{U}^0,e\in \mathcal{U}^1 \text{ and } g\in G\}$$
be a $(G,\mathcal{U})$-family in a $C^*$-algebra $\mathcal{A}$ as in Definition \ref{defn2...3}. Then, for every $g\in G$, the series $\sum_{v\in U^0}V_{\{v\},g}$ converges to an element $V_g$ in the multiplier algebra $M(\mathcal{A})$. Moreover, we have
\begin{enumerate}[$(1)$]
  \item the map $g \mapsto V_g$ is a unitary $*$-representation of $G$ into $M(\mathcal{A})$, and furthermore,
  \item \begin{enumerate}[(a)]
  \item $V_g P_A=P_{g \cdot A} V_g$ for all $A\in \mathcal{U}^0$, and
  \item $V_g S_e= S_{g \cdot e} V_g$ for all $e \in \mathcal{U}^{1}$.
        \end{enumerate}
\end{enumerate}
\end{prop}

\begin{proof}
It is known that there exists a Hilbert space $H$  such that $M(\mathcal{A}) \cong B(H)$, where $B(H)$ is the bounded operator space on $H$. Without loss of generality, assume $M(\mathcal{A})=B(H)$ and $\mathcal{A}$ is a (closed) subalgebra of $B(H)$. Let $\mathcal{D}$ be the $C^*$-subalgebra of $B(H)$ generated by the family $\{S_e, P_A, V_{A,g}:A\in\mathcal{U}^0,e\in \mathcal{U}^1,g\in G \}$ and write $K=\overline{\mathcal{D}(H)}$ as a closed subalgebra of $H$.

List the vertex set $U^0=\{v_1,v_2, \ldots \}$ and set $A_n=\{v_1, \ldots, v_n\}$ for each $n \geq 1$. Since
$$ V_{A_n ,g}=\sum_{i=1}^n V_{\{v_i\}, g}$$
by (CK1) of Definition \ref{defn2.3}, we will prove that the sequence $\{ V_{A_n ,g}\}_{n\geq 1}$ is convergent in $M(\mathcal{A})$ for every $g\in G$.

First, for a given element $\xi \in K$ and $g\in G$, we demonstrate that $\{V_{A_n,g}(\xi)\}_{n \geq 1}$ is a Cauchy sequence in $H$, and hence it is convergent. To see this, since $\xi \in K=\overline{\mathcal{D}(H)}$, there exists a sequence $\{T_k(\xi_k)\} \subseteq \mathcal{D}(H)$ with $\{T_k\}\subseteq \mathcal{D}$ such that $\xi= \lim_{k\rightarrow \infty}T_k(\xi_k)$. Let $\epsilon>0$. Then, there exists $N\geq 1$ such that $\|\xi- T_N(\xi_N)\|< \epsilon /3$. Observe that
$$\|V_{A,g}\|=\|V_{A,g} V_{A,g}^*\|^{1/2} \overset{\text{Def. } \ref{defn2...3}(3)}{=} \|V_{A,g} V_{g^{-1}\cdot A,g^{-1}}\|^{1/2}=\|V_{A,1_G}\|^{1/2}=\|P_A\|\leq 1,$$
and hence we may write for each $n\geq m\geq 1$,
\begin{align}\label{eq6.1}
\|V_{A_n,g} T_N - V_{A_m,g}T_N\|&=\|V_{A_n,g}P_{g^{-1}\cdot A_n} T_N -V_{A_n,g} P_{g^{-1}\cdot A_m}T_N\| \nonumber\\
    & \hspace{4cm} (\text{as } V_{A_n,g} P_{g^{-1}\cdot A_m}=V_{A_m,g}) \nonumber \\
&\leq \|V_{A_n,g}\| \|P_{g^{-1} \cdot A_n} T_N -P_{g^{-1} \cdot A_m} T_N\| \nonumber \\
&\leq \|P_{g^{-1} \cdot A_n} T_N -P_{g^{-1} \cdot A_m} T_N\|.
\end{align}
Since $\{P_{A_n}\}_{n\geq 1}$ is an approximate identity for the subalgebra $\mathcal{D}=\langle S_e,P_A,V_{A,g}\rangle$, and so is $\{P_{g^{-1} \cdot A_n}\}_{n \geq 1}$ for any $g\in G$, there exists $n_0 \geq 1$ such that for every $n \geq m\geq n_0$,
$$\|V_{A_n,g} T_N -V_{A_m,g}T_n\| \overset{(\ref{eq6.1})}{\leq} \|P_{g^{-1} \cdot A_n} T_N -P_{g^{-1} \cdot A_m} T_N\| < \frac{\epsilon}{3(\|\xi_N\|+1)},$$
and therefore, we get
\begin{align*}
\|V_{A_n,g}(\xi) - V_{A_m,g}(\xi)\| & \leq \|V_{A_n,g}(\xi) - V_{A_n,g}(T_N(\xi_N))\|  \\
 & \hspace{1.5cm} + \|V_{A_n,g}(T_N(\xi_N))-V_{A_m,g}(T_N(\xi_N))\| \\
 & \hspace{1.5cm} + \|V_{A_m,g}(T_N(\xi_N))-V_{A_m,g}(\xi)\|\\
& \leq \|V_{A_n,g}\| \|\xi-T_N(\xi_N)\|+\frac{\epsilon \|\xi_N\|}{3(\|\xi_N\|+1)} \\
& \hspace{3.7cm} +\|V_{A_m,g}\|\|T_N(\xi_N)-\xi\|\\
& < (1\cdot \frac{\epsilon}{3})+ \frac{\epsilon}{3}+ (1 \cdot \frac{\epsilon}{3}) =\epsilon.
\end{align*}
Consequently, the sequence $\{V_{A_n,g}(\xi)\}_{n\geq 1}$ is Cauchy in $H$ for all $\xi\in K$, and therefore, it is convergent.

Now, if we define
\begin{equation}\label{eq6.2}
V_g(\xi):= \lim_{n \rightarrow \infty} V_{A_n , g} (\xi_1)= \lim_{n \rightarrow \infty} V_{A_n , g} (\xi)
\end{equation}
for every $g\in G$ and $\xi=\xi_1 + \xi_2\in K\oplus K^\perp=H$ (recall that $V_{A_n,g}|_{K^\perp}=0$ by the definition of $K$), then $V_g$ belongs to $B(H)$ and $\{V_{A_n,g}\}_{n\geq 1}$ converges to $V_g$ in the SOT.

Statements (1) and (2) follow easily from relations of Definition \ref{defn2...3}. Indeed, for each $g\in G$, one may compute
\begin{align*}
V_g V_g^*=\lim_{n\rightarrow \infty} V_{A_n,g} V_{A_n,g}^*&=\lim_{n\rightarrow \infty} V_{A_n,g} V_{g^{-1}\cdot A_n,g^{-1}}\\
&=\lim_{n\rightarrow \infty} V_{(A_n\cap g\cdot A_n),1_G}=\lim_{n\rightarrow \infty} V_{A_n,1_G}V_{g\cdot A_n,1_G}=V_{1_G}
\end{align*}
because $\cup_{n=1}^\infty g\cdot A_n=U^0$. Similarly, we have $V_g V_h=V_{gh}$ for $g,h\in G$, and hence, the map $g\mapsto V_g$ is a unitary $*$-representation from $G$ into $M(\mathcal{O}_{G,\mathcal{U}})$.

For statement (2), given any $A\in \mathcal{U}^0$ and $g\in G$, relation (4) of Definition \ref{defn2...3} yields
$$
V_g P_A=\lim_{n\rightarrow \infty} V_{A_n,g} P_A=\lim_{n\rightarrow \infty} V_{(A_n \cap g \cdot A), g}=\lim_{n\rightarrow \infty} P_{g \cdot A} V_{A_n,g}=P_{g\cdot A}V_g,
$$
which is property (a). Property (b) can be obtained analogously, so we are done.
\end{proof}

In particular, we conclude that:

\begin{cor}
Let $(G,\mathcal{U},\varphi)$ be a self-similar ultragraph. Then the $C^*$-algebra $\widetilde{O}_{G,\mathcal{U}}$, defined in Definition \ref{defn6.1} above, is isomorphic to $\mathcal{O}_{G,\mathcal{U}}$.
\end{cor}

\begin{proof}
Suppose that $\{s_e,p_A,u_{A,g}\}$ and $\{\widetilde{s}_e,\widetilde{p}_A,\widetilde{u}_g\}$ are generating families, as in Definitions \ref{defn2...3} and \ref{defn6.1}, for $\mathcal{O}_{G,\mathcal{U}}$ and $\widetilde{\mathcal{O}}_{G,\mathcal{U}}$ respectively. By Proposition \ref{prop6.2}, we may define the unitaries $u_g:=\sum_{v\in U^0}u_{\{v\},g}$ in $M(\mathcal{O}_{G,\mathcal{U}})$ for all $g\in G$, so that the family $\{s_e,p_A,u_g:A\in \mathcal{U}^0,e\in \mathcal{U}^1,g\in G\}$ in $M(\mathcal{O}_{G,\mathcal{U}})$ satisfies relations of Definition \ref{defn6.1}. On the other hand, if we set $\widetilde{u}_{A,g}:=\widetilde{p}_A \widetilde{u}_{g}$ for $A\in \mathcal{U}^0$ and $g\in G$, then $\{\widetilde{s}_e,\widetilde{p}_A,\widetilde{u}_{A,g}\}$ is a $(G,\mathcal{U})$-family in  $\widetilde{\mathcal{O}}_{G,\mathcal{U}}$ in the sense of Definition \ref{defn2...3}. Hence, by universality, there are canonical $*$-homomorphisms $\phi:\widetilde{\mathcal{O}}_{G,\mathcal{U}} \rightarrow \mathcal{O}_{G,\mathcal{U}}$ and $\psi:\mathcal{O}_{G,\mathcal{U}}\rightarrow \widetilde{\mathcal{O}}_{G,\mathcal{U}}$ such that $\phi \circ \psi$ and $\psi \circ \phi$ are identity on the generating terms in the spans of (\ref{(2.1)}) and (\ref{eq6.1-1}), respectively. It follows that $\phi$ and $\psi$ are isomorphisms, concluding the result.
\end{proof}


\subsection{A crossed product form}\label{subs6.2}
Let us recall here the definition of a crossed product. A {\it{discrete $C^*$-dynamical system}} (or briefly {\it{discrete dynamical system}}) is a triple $({\mathcal{A}},G,\eta)$ consisting of a $C^*$-algebra ${\mathcal{A}}$, a discrete group $G$, and a $*$-homomorphism $\eta:G\rightarrow \text{Aut}({\mathcal{A}})$, denoted by $g \mapsto \eta_g$. Then, $C_c(G,{\mathcal{A}})$ is the linear span of finitely supported ${\mathcal{A}}$-valued functions on $G$. A typical element $f$ in $C_c(G,{\mathcal{A}})$ is written as a sum $f=\sum_{g\in G}a_g \delta_g$, with $a_g \in {\mathcal{A}}$, such that only finitely many coefficients $a_g$ are nonzero. We equip $C_c(G,{\mathcal{A}})$ with the $\eta$-twisted convolution

\begin{equation}\label{eq6.4}
f_1 *_{\eta}f_2=\sum_{g,h\in G}a_g {\eta}_g (b_h)\delta_{gh},
\end{equation}
where $f_1=\sum_{g\in G}a_g \delta_g$ and $f_2=\sum_{h\in G}b_h \delta_h$, and the $*$-operation
\begin{equation}\label{eq6.5}
f^*=\sum_{g\in G}\eta_{g^{-1}}(a^*_g) \delta_{g^{-1}}.
\end{equation}

\begin{defn}
The ({\it{full}}) {\it{crossed product}} of a $C^*$-dynamical system $({\mathcal{A}},G,\eta)$, denoted by ${\mathcal{A}}\rtimes_{\eta} G$, is the completion of $C_c(G,{\mathcal{A}})$ taken by the norm
$$\| x\| _{u}=\sup \|   \pi(x)   \|,$$
where supremum is over all cyclic $*$-homomorphisms $\pi:C_c(G,{\mathcal{A}})\rightarrow B(H)$.
\end{defn}

Let $(G,\mathcal{U},\iota)$ be a self-similar ultragraph such that $\iota:G \times {\mathcal{U}}^1 \rightarrow G$ is the trivial $1$-cocycle defined by $\iota(g,\alpha)=g$ for all $\alpha \in {\mathcal{U}}^1$. In particular, $(G,\mathcal{U},\iota)$ is pseudo free in this case \cite[Theorem 10.4]{ana02}.

In the rest of the paper, we fix a generating $(G,\mathcal{U})$-family $\{s_e,p_A,u_{A,g}:A\in \mathcal{U}^0,e\in \mathcal{U}^1 \text{ and } g\in G\}$ for $\mathcal{O}_{G,\mathcal{U}}$. Observe that an application of the gauge-invariant uniqueness theorem for ultragraph $C^*$-algebras \cite[Proposition 5.5]{ana00} implies that the ultragraph $C^*$-algebra $C^*(\mathcal{U})$ can be canonically embedded in $\mathcal{O}_{G,\mathcal{U}}$. So, with respect to this embedding, we may regard the ultragraph $C^*$-algebra $C^*(\mathcal{U})$ as a $C^*$-subalgebra of $\mathcal{O}_{G,\mathcal{U}}$ so that the Cuntz-Krieger $\mathcal{U}$-family $\{s_e, p_A:A\in \mathcal{U}^0,e\in \mathcal{U}^1\}$ generates $C^*(\mathcal{U})$. Thus, one can define an action $\eta:G \curvearrowright C^*({\mathcal{U}})$ by
$$\eta_g(s_\alpha p_As^*_\beta)=s_{g\cdot \alpha}p_{g\cdot A}s^*_{g\cdot \beta}$$
on each term of (\ref{eq2.2}).

\begin{lem}\label{lem6.2}
Given $g\in G$, let $u_g:=\sum_{v\in U^0}u_{\{v\},g}$ be an element in the multiplier algebra $M(\mathcal{O}_{G,{\mathcal{U} }})$ by Proposition \ref{prop6.2}. Then
\begin{equation}\label{eq6.6}
\eta_g(a)=u_g a u^*_g
\end{equation}
for every $g\in G$ and $a\in C^*(\mathcal{U}).$
\end{lem}

\begin{proof}
First, since the 1-cocycle is $\iota(g,\alpha)=g$, Proposition \ref{prop6.2}(2) implies that
$u_g s_\alpha = s_{g \cdot \alpha} u_g$ and $u_g p_A=p_{g \cdot A} u_g$
for all $g\in G$, $A\in \mathcal{U}^0$, and $\alpha \in \mathcal{U}^{\geq 1}$. Thus, we may conclude the result by verifying (\ref{eq6.6}) for elements of the form $a=s_\alpha p_A s^*_\beta$ in (\ref{eq2.2}) as follows:
\begin{align*}
u_g a u^*_g &=u_g (  s_\alpha p_A s^*_\beta  ) u^*_g \\
&=(u_g s_\alpha) p_A (u_g s_\beta)^* \\
&=(s_{g\cdot \alpha} u_g) p_A (s_{g\cdot \beta}u_g)^* \\
&=(s_{g\cdot \alpha} p_{g \cdot A}) u_g u_{g^{-1}} s^*_{g\cdot \beta} \\
&=s_{g\cdot \alpha} p_{g \cdot A} s^*_{g\cdot \beta} \\
&=\eta_g (a).
\end{align*}
\end{proof}

It follows from Lemma \ref{lem6.2} that $(G,C^*({\mathcal{U}}),\eta)$ is a dynamical system and  we can construct the crossed product $C^*({\mathcal{U}})\rtimes_{\eta}G$. Note that, as we assume $C^*({\mathcal{U}}) \subseteq {\mathcal{O}}_{G,{\mathcal{U}}}$, a typical generator $s_\alpha p_A s_\beta^* \delta_g$ in $C_c(G, C^*(\mathcal{U}))$ is corresponding to the one $s_\alpha p_A s_\beta^* u_g$ in $\mathcal{O}_{G,\mathcal{U}}$, and the next Theorem ensures $C^*({\mathcal{U}})\rtimes_{\eta}G \cong {\mathcal{O}}_{G,{\mathcal{U}}}$ via this correspondence. Its proof is analogous to that of \cite[Theorem 8.4]{ana02}. Before that, let us recall the definition of tight representations from \cite{exe21}.

\begin{defn}[{\cite[Definition 13.1]{{exe21}}}]\label{defn6.4}
Let $\mathcal{S}$ be an inverse semigroup with a zero element. A representation $\pi:\mathcal{S}\rightarrow \mathcal{A}$ on a unital $C^*$-algebra $\mathcal{A}$ is called \emph{tight} if for every $X,Y\subseteq \mathcal{E}(\mathcal{S})$ and every finite cover $Z$ for the set
$$\mathcal{E}(\mathcal{S})^{X,Y}:=\left\{e\in \mathcal{E}(\mathcal{S}): e\leq f, \;\; \mathrm{for \;all}\;\; f\in X, ~ \mathrm{and} ~ ef'=0 \;\; \mathrm{for \;all}\;\; f'\in Y\right\},$$
one has
$$\bigvee_{e\in Z} \pi(e)=\bigwedge_{f\in X} \pi(f) \wedge \bigwedge_{f'\in Y}(1-\pi(f')).$$
In addition, we say $\pi$ is a \emph{universal tight representation} if for every tight representation $\pi:\mathcal{S}\rightarrow \mathcal{A}$, there exists a $*$-homomorphism $\psi: \pi(\mathcal{S})\rightarrow \mathcal{A}$ such that $\psi \circ \pi=\phi$.
\end{defn}

\begin{thm}\label{thm6.3}
Let $(G,\mathcal{U},\iota)$ be a self-similar ultragraph with the trivial $1$-cocycle $\iota$. Let $\{s_e,p_A\}$ be a Cuntz-Krieger $\mathcal{U}$-family generating $C^*(\mathcal{U})$. Then the map
$$\phi:{\mathcal{S}}_{G,{\mathcal{U}}}\rightarrow C^*({\mathcal{U}})\rtimes_{\eta}G,$$
defined by $\phi(\alpha,A,g,\beta)=s_\alpha p_A s^*_{g\cdot \beta}  \delta_g$
\footnote{As $\omega$ is the universal 0-length path, in the case $\alpha=\omega$, we have $\phi(\omega,A,g,\beta)=p_A s_{g\cdot \beta}^* \delta_g$; and a same thing can be said for the case $\beta=\omega$.},
is a universal tight representation of ${\mathcal{S}}_{G,{\mathcal{U}}}$. Consequently, $C^*(\mathcal{U}) \rtimes_\eta G \cong \mathcal{O}_{G, \mathcal{U}}$.
\end{thm}

\begin{proof}
Using the multiplication and inversion in $\mathcal{S}_{G,\mathcal{U}}$ and those (\ref{eq6.4}), (\ref{eq6.5}) in $C^*(\mathcal{U})\rtimes_\eta G$, it is straightforward to verify that $\phi$ is a $*$-homomorphism. Also, for each idempotent $q_{(\alpha,A   )}\in {\mathcal{E}}( {\mathcal{S}}_{G,{\mathcal{U}}})$ we have $\phi( q_{(\alpha,A   )})=s_{\alpha}p_A s^*_{\alpha} \delta_{1_G}$, and the tightness of $\phi$ may be shown by an argument analogous to that for the map in \cite[Theorem 8.2]{ana02}.

So, we prove the universality of $\phi$ in the sense that if $\psi:{\mathcal{S}}_{G,{\mathcal{U}}}\rightarrow {\mathcal{A}}$ is a tight representation into a $C^*$-algebra ${\mathcal{A}}$, then there exists a $*$-homomorphism $T:C^*({\mathcal{U}})\rtimes_{\eta}G \rightarrow {\mathcal{A}}$ such that $T \circ \phi=\psi$. Assume such a representation $\psi:{\mathcal{S}}_{G,{\mathcal{U}}}\rightarrow {\mathcal{A}}$ is given. If we define
$$S_e=\psi(e,s(e),1_G,\omega)\text{ and }P_A=\psi(\omega,A,1_G,\omega)$$
for $A\in {\mathcal{U}}^0$ and $e\in {\mathcal{U}}^1$, then \cite[Theorem 8.3]{ana02} shows that $\{ s_e,p_A: \; A\in {\mathcal{U}}^0  , e\in {\mathcal{U}}^1 \}$ is a Cuntz-Krieger ${\mathcal{U}}$-family. Thus, there is a homomorphism $\pi:C^*({\mathcal{U}})\rightarrow {\mathcal{A}}$ such that $\pi(p_v)=P_A$ and $\pi(s_e)=S_e$ for all $A\in {\mathcal{U}}^0$ and $e\in {\mathcal{U}}^1$.

Moreover, define
$$V_{v,g}:=\phi(\{v\} , \{v\},g,g^{-1}\cdot \{v\})$$
for every $v\in U^0$ and $g\in G$. If $V_g:=\sum_{v\in U^0}V_{v,g}$ in the multiplier algebra $M({\mathcal{A}})$, then $V:G  \rightarrow M({\mathcal{A}})$, by $g \mapsto V_g$, is a unitary $*$-representation of $G$. (We may follow the proof of Proposition \ref{prop6.2} to verify that $V_g$ is well-defined.) In order to see this, for $g,h \in G$, one may compute
\begin{align*}
V_g V_h&=(\sum_{v\in U^0}V_{v,g})(\sum_{w\in U^0}V_{w,h}) \\
&=\sum_{\substack{v,w\in U^0\\g^{-1}\cdot v=w\\ }}\phi((\{v\} , \{v\},g,g^{-1}\cdot \{v\})(\{w\} , \{w\},h,h^{-1}\cdot \{w\}))\\
&=\sum_{v\in U^0}\phi(\{v\} , \{v\},gh,h^{-1}g^{-1}\cdot \{v\})\\
&=\sum_{v\in U^0}V_{v,gh}\\
&=V_{gh},
\end{align*}
hence $V$ is multiplicative. Similarly, we have $V_{g^{-1}}=V^*_g$ and therefore $V$ is a unitary $*$-representation.

Next, $(V,\pi)$ induces a linear map $V \times \pi$ of $C_c(G,C^*(\mathcal{U}))$ via the definition
$$ V \times \pi(\sum_{g\in G}a_g \delta_g)=\sum_{g\in G}\pi(a_g)V_g.$$
Note that $(V,\pi)$ is covariant as well because
\begin{align*}
V_g \pi(s_\alpha p_A s^*_\beta)V^*_g&=(V_g S_\alpha) P_A (V_g S_\beta)^* \\
&=(S_{g\cdot \alpha} V_g) P_A (S_{g\cdot \beta} V_g)^*\\
&=S_{g\cdot \alpha}P_{g\cdot A}(V_g V_{g^-1})S^*_{g\cdot \beta}\\
&=S_{g\cdot \alpha}P_{g\cdot A}S^*_{g\cdot \beta} \\
&=\pi(\eta_g(s_\alpha p_A s^*_\beta))
\end{align*}
for the elements of the form $a=s_\alpha p_A s^*_\beta$ in (\ref{eq2.2}). Therefore, by \cite[Proposition 4.1.3]{brw40}, we have a $*$-homomorphism $T:C^*({\mathcal{U}}) \rtimes_\eta G \rightarrow {\mathcal{A}}$ such that $T |_{C_c(G,C^*({\mathcal{U}}) )}=V \times \pi$. Since $T  \circ \phi=\psi$, we conclude that $\phi$ is a universal tight representation.

The second statement follows from the first one together with \cite[Theorem 8.4]{ana02}.
\end{proof}

As said before, if $(G,{\mathcal{U}},\varphi) $ is a self-similar ultragraph with $\varphi=\iota$, then it is pseudo free or equivalently, there are no paths $\alpha \in {\mathcal{U}}^{\geq 1}$ being strongly fixed by some $g\in G \backslash \{1_G\}$. So, by considering property $(1)$ in Condition $(*)$, $(G,{\mathcal{U}},\iota) $ satisfies Condition $(*)$ if and only if for each $v\in U^0$ and $g\in G \backslash \{1_G\}$, there is $x \in {\mathcal{U}}^\infty$ with $r(x)=\{v\}$ such that $g\cdot x  \neq x$ (this is equivalent to $g\cdot \alpha \neq \alpha$ for some $\alpha \in {\mathcal{U}}^{\geq 1}$ with $r(\alpha)=\{v\}$). Thus, we conclude the following result for the simplicity of ${\mathcal{O}}_{G,{\mathcal{U}}}$.

\begin{cor}\label{cor6.4}
Let $(G,{\mathcal{U}},\iota) $ be a self-similar ultragraph with the trivial $1$-cocycle $\iota$. Suppose that $G$ is an amenable group. If
\begin{enumerate}[$(1)$]
\item $(G,{\mathcal{U}},\iota) $ is $G$-cofinal,
\item every  $G$-cycle in $(G,{\mathcal{U}},\iota) $ has an entrance, and
\item for every $v\in U^0$ and $g\in G \backslash \{1_G\}$, there is $\alpha \in {\mathcal{U}}^{\geq 1}$ with $r(\alpha)=\{v\}$ such that $g\cdot \alpha \neq \alpha$,
\end{enumerate}
then the $C^*$-algebra ${\mathcal{O}}_{G,{\mathcal{U}}}    \cong   C^*({\mathcal{U}})\rtimes_{\eta}G$ is simple.
\end{cor}

\begin{proof}
It is known that the ultragraph $C^*$-algebra $C^*( {\mathcal{U}}  )$ is nuclear (this follows from \cite[Theorem 5.2]{li19-boundary} and \cite[Proposition 6.1]{kat40} for example).
As $G$ is an amenable group, then ${\mathcal{O}}_{G,{\mathcal{U}}} \cong C^*( {\mathcal{U}}  )\rtimes_\eta G$ is a nuclear $C^*$-algebra as well.

On the other hand, since $(G,{\mathcal{U}},\iota) $ is pseudo free, the groupoid $\mathcal{G}_{\mathrm{tight}}({S}_{G,{\mathcal{U}}})$ is Hausdorff \cite[Theorem 9.5]{ana02}, and also is minimal and effective by Theorems \ref{thm3.4} and \ref{thm4.8}. Therefore, \cite[Theorem 5.1]{brt92} implies that ${\mathcal{O}}_{G,{\mathcal{U}}} \cong C^*( \mathcal{G}_{\mathrm{tight}}({S}_{G,{\mathcal{U}}}))$ is simple.
\end{proof}

For example, if we consider the self-similar ultragraphs in Examples \ref{ex5.1} and \ref{ex5.2} with the trivial 1-cocycle $\varphi=\iota$, then they satisfy the above conditions (1), (2) and (3). Since both $\mathbb{Z}$ and $\mathbb{Z}_2$ are amenable groups, Corollary \ref{cor6.4} implies that the associated $C^*$-algebras are simple.

Note that the self-similar ultragraph $(\mathbb{Z},{\mathcal{U}},\iota) $ in Example \ref{ex5.3}, with $\varphi=\iota$, does not satisfy Condition $(*)$ by Proposition \ref{prop5.5} (we have $\iota(1,e_0)+\iota(1,e_1)=1+1=2$). In particular, Corollary \ref{cor6.4} above could not be applied for this self-similar ultragrah.


\subsection*{Acknowledgement}
The authors would like to thank the anonymous referees for their helpful and detailed comments to improve the initial version of the manuscript. They also acknowledge financial support from Shahid Chamran University of Ahvaz (grant no. SCU.MM1403.279).




\end{document}